\documentclass{amsart}

\usepackage{amsmath}
\usepackage{amsfonts}

\usepackage{verbatim}
\usepackage{tikz}

\usepackage{amssymb,amsthm,amscd}
\usepackage[OT4]{fontenc}

\newtheorem{theorem}{Theorem}
\newtheorem{lemma}[theorem]{Lemma}

\newtheorem{corollary}[theorem]{Corollary}
\newtheorem{proposition}[theorem]{Proposition}

\theoremstyle{definition}

\newtheorem{observation}[theorem]{Remark}

\numberwithin{theorem}{section}
\numberwithin{equation}{section}

\DeclareMathOperator\co{{\it\mathbb C^n}}

\DeclareMathOperator\om{{\it\omega}}

\DeclareMathOperator\al{{\it\alpha}}

\DeclareMathOperator\de{{\it\delta}}


\begin{document}

\title[Stability and regularity of solutions of Monge-Amp\`ere  equations]{Stability and regularity of solutions of the  Monge-Amp\`ere  equation on Hermitian manifolds} 
\author[S. Ko\l odziej and N.-C. Nguyen]{S\l awomir Ko\l odziej and Ngoc Cuong Nguyen} 
\date{}
\maketitle

\bigskip

\begin{abstract}
We prove stability of solutions of the complex Monge-Amp\`ere equation on compact Hermitian manifolds, when the right hand side varies in a bounded set in $L^p , p>1$
and it is  bounded away from zero. Such solutions are shown to be H\"older continuous. As an application we extend a recent result of  Sz\'ekelyhidi and Tosatti  
on K\"ahler-Einstein equation  from K\"ahler to  Hermitian manifolds.


\end{abstract}

\section{introduction}
\rm 
Ever since the solution of the Calabi conjecture by S.T. Yau \cite{Y}  the complex Monge-Amp\`ere equation played a prominent role
in complex geometry. In the last decades the weak solutions of the equation  on compact K\"ahler manifolds also found many applications
in the study of degenerations of canonical metrics and the limits of the K\"ahler-Ricci flow. In those settings, often, when a family
 of   K\"ahler metrics approaches the boundary of the K\"ahler cone their volume forms blow up, becoming unbounded, but still
remain bounded in $L^p$ for some $p>1.$ Then the stability estimates \cite{kolodziej03} \cite{kolodziej05} provide good control of the potentials 
of those metrics close to the singularity set. In particular the potentials are then H\"older continuous (see \cite{kolodziej08}, 
\cite{demailly-et-al}). 
Those results found  a number of applications in  the works on the  K\"ahler-Ricci flow  of Tian-Zhang \cite{tian-zhang06}, and Song-Tian           
\cite{ST1, ST2, ST3};  in Tosatti's description of the limits of families of Calabi-Yau metrics when the K\"ahler class
degenerates \cite{To1}, \cite{To2}; and in the  recent solution of the Donaldson-Tian-Yau conjecture. The proof of Chen-Donaldson-Sun  \cite{CDS2}  uses stability and H\"older continuity results  for approximation of cone K\"ahler-Einstein metrics by the smooth ones.
The pluripotential approach was further developed in the papers of  Eyssidieux-Guedj-Zeriahi \cite{EGZ09}, of those authors together  with
 Boucksom \cite{BEGZ}, and in several other articles. An up-to-date account of those developments can be found in the survey of  Phong,  Song and Sturm \cite{pss12}.

In this paper we are concerned with the complex Monge-Amp\`ere equation on compact Hermitian (non-K\"ahler) manifolds. In the eighties Cherrier \cite{cherrier87} made an attempt to prove the analogue of the Calabi-Yau theorem in this setting obtaining the result under a rather restrictive assumption. Further progress was made in \cite{GL10} and \cite{TW10a}.  Finally Tosatti and Weinkove \cite{TW10b} gave the complete proof. Those results came amid  a considerable growth of research activity in Hermitian geometric analysis.
In this context one should mention a paper by Fu and Yau \cite{FY08} with a construction of some non-trivial solutions of the Strominger system. In relation to this problem
Fu-Wang-Wu \cite{fww1} introduced, form-type Calabi-Yau equation and solved it \cite{fww2}  for metrics of nonnegative orthogonal bisectional curvature.
Tosatti and Weinkove \cite{TW13} were able to solve the equation without the extra assumption. In the equation there are two reference metrics. If both are non-K\"ahler,
the solution, which gives the confirmation of the Gauduchon conjecture (analogue of the one of Calabi for Gauduchon metrics), was found by  Sz\'ekelyhidi, Tosatti and Weinkove \cite{SzTW}.
Popovici \cite{popovici13}  considered a variant of this equation studying  moduli spaces of Calabi-Yau $\partial\bar{\partial}$-manifolds. 
The  Monge-Amp\`ere equation  is closely related to  the Chern-Ricci flow, intensively studied recently in the  papers of Gill \cite{gill11, gill13}, Tosatti-Weinkove \cite{TW12a, TW12b}, Tosatti-Weinkove-Yang \cite{TWYang13}. 
Chiose \cite{chiose13} and the second named author \cite{cuong15}  used the solutions of the equation to prove a conjecture of Demailly and Paun \cite{DP04}  in some special cases.

The analysis of geometrically meaningful equations on Hermitian manifolds is
harder than in the K\"ahler case due to the torsion terms which are difficult to handle in the estimates. This accounts for a somewhat technical character of the proofs.

\it Results. \rm
Throughout the paper $(X, \omega )$ will denote a compact manifold  $X$ of dimension $n>1$, with a Hermitian metric $\omega .$
S. Dinew and the first author obtained in \cite{DK12} $L^{\infty }$ a priori estimates for the complex Monge-Amp\`ere equation
$$
	(\omega + dd^c u)^n =const.  f\omega^n, 
$$
with the nonnegative  right hand side in $ L^p(\omega^n)$, $p>1.$ The weak continuous solutions, under the same assumption,  were obtained
by the authors in \cite{KN}. Up till now the uniqueness of solutions, normalised, for instance,  by $\sup _X u=0$, has not been established.
It follows, for strictly positive $f$, from the main result of this paper, which is the following stability statement.

\medskip
\noindent {\bf Theorem A.} {\em 
Let $0\leq f, g\in L^p(\omega^n)$, $p>1$, be such that $\int_X f \omega^n >0$, $\int_X g\omega^n>0$. Consider two continuous $\omega$-psh solutions of the complex Monge-Amp\`ere equation 
\[
	(\omega + dd^c u)^n = f\omega^n, \quad 
	(\omega + dd^c v)^n = g\omega^n,
\] 
with 
$
	\sup_X u = \sup_X v =0.
$
Assume that
\[
	 f \geq c_0>0   \ \ (c_0   \ \ \text{a constant}).
\] 
Fix $0< \alpha < \frac{1}{n+1}$. Then, there exists $C = C(c_0, \alpha, \|f\|_p, \|g\|_p)$ such that
\[
	\|u-v\|_\infty \leq C \|f-g\|_p^\alpha.
\]
}

As compared to the corresponding theorem for K\"ahler manifolds we have here an additional hypothesis that $f$ is nondegenerate ($f>c_0 >0$).
It would be very desirable to remove or weaken it. Note, however, that on non-K\"ahler manifolds the notion of  "stability" itself has an extra dimension, since $\int _X f\om ^n$ is no longer fixed. It means that if we consider a small  perturbation $\tilde{f}$ of $f$ on the right hand side   then the Monge-Amp\`ere equation has a solution for $c\tilde{f}$ and this constant is not determined by $\om .$ 
On the bright side, the H\"older exponent in Theorem A is independent of $(X, \omega )$ and it is almost as good as in the K\"ahler case (see \cite{dinew-zhang10}), except that we consider
$L^p $ norm of $(f-g)$ instead of $L^1$ norm. We next prove the result with $L^1$ norm in Theorem \ref{stability-l1}, but then the exponent is worse.
The proof of Theorem A required a completely new method. 

With the stability at our disposal we could  prove two other theorems. First, the H\"older continuity of solutions of the Monge-Amp\`ere equation on compact Hermitian manifolds.

\medskip
\noindent {\bf Theorem B.} {\em 
 Consider the  solution $u$ of the complex Monge-Amp\`ere equation
$$\om _u^n = f\om ^n , $$
on  $(X, \om )$  a compact  Hermitian manifold, with $ f > c_0 > 0, \|f\|_p <\infty.$
Then for any    $\alpha < \frac{2}{p^* n (n+1)+1}$    the function $u$
is H\"older continuous, with H\"older exponent $\alpha .$
}
\medskip

Again we have the hypothesis that $f$ is nondegenerate, which is not needed on K\"ahler manifolds. The H\"older exponent is worse than in the K\"ahler counterpart,  (roughly) by factor $1/n.$

The second application of the stability result is an extension of a theorem of 
 Sz\'ekelyhidi and Tosatti  \cite{SzTo11} to the case of   compact Hermitian manifolds.

\medskip
\noindent {\bf Theorem C.} {\em
Let $(X,\omega)$ be a compact $n$-dimensional Hermitian manifold. 
Suppose that $u\in PSH(\omega)\cap L^\infty(X)$ is a
solution of the equation
\[ 
	(\omega + dd^c u)^n = e^{-F(u,z)}\omega^n
\]
in the weak sense of currents, where $F:\mathbb{R}\times X\to \mathbb{R}$ is smooth. 
Then $u$ is smooth.
} 
\medskip
In \cite{SzTo11} the authors obtained this for K\"ahler manifolds to derive that
if $X$ is a Fano manifold and  $\omega$ represents the first Chern class, then any  K\"ahler-Einstein current with bounded potentials is  smooth.
 Nie [Nie13] recently generalised this result to special compact Hermitian manifolds, and observed that her  higher order estimates
and a stability result would give the theorem above. We have just provided this result.

\bigskip

{\bf Acknowledgement.} The authors would like to thank S\l awomir Dinew for useful comments, in particular for a quicker proof of Lemma \ref{bbg}. We are also grateful to anonymous referees for the comments  and suggestions which improved the exposition. The research was supported by NCN grant 2013/08/A/ST1/00312. 
\section{preliminaries}
\label{S1}

In this section we give some auxiliary results used later for the proof of the stability estimates and uniqueness in Section~\ref{S2}. 
Some of them are interesting in themselves. 

Let $(X, \omega)$ be a compact $n$-dimensional Hermitian manifold.  $PSH(\omega)$ stands for the set of all $\omega$-plurisubharmonic ($\omega$-psh) functions on $X$. If $u\in PSH(\omega)$, then 
\[
	\omega_u:= \omega + dd^c u\geq 0. 
\]
We shall denote throughout, by $B$
the "curvature" constant $B>0$  satisfying
\begin{equation}
\label{curv-const}
	-B \omega^2 \leq dd^c \omega \leq B \omega^2, \quad
	-B \omega^3 \leq d\omega \wedge d^c \omega \leq B \omega^3.
\end{equation}
For $r\geq 1$, its conjugate number will be denoted by $r^*$, and we write $L^r(\omega^n)$ for $ L^r(X, \omega^n)$, and
\[
	\|.\|_r = \left(\int_X |.|^r \omega^n \right)^\frac{1}{r},  \quad
	\|.\|_\infty = \sup_X |.|.
\]
The volume of a Borel set $E\subset X$ with respect to the metric $\omega$ is given by
\[
	Vol_\omega(E) := \int_E \omega^n,
\]
and without loss of generality we normalise the volume of $X$ to be $1$, i.e.
\[
	Vol_\omega(X) = \int_X \omega^n =1.
\]
We always denote by $C$ a  uniform positive constant  that depends only on $X, \omega$ and the dimension $n$. It  may differ from place to place.

\medskip

Because  $\omega$  is not closed, the Monge-Amp\`ere operator does not preserve the volume of the manifold. 
Therefore, given $f$ on the right hand side of the  Monge-Amp\`ere equation the solution possibly exists for $cf$ ($c$ a positive constant).
If $c=1$ then we call $f$ \it MA-admissible. \rm  
 Thanks to results in \cite{KN}, \cite{cuong15} we are able to get {\em a priori} bounds for the constant $c$. 

\begin{lemma}
\label{const-comp}
If $u, v \in PSH(\omega) \cap  L^\infty(X)$ satisfy
\[
	\omega_u^n \leq  c \; \omega_v^n
\]
for some $c>0$, then $c \geq 1$.
\end{lemma}

\begin{proof}
See \cite[Lemma 2.3, Corollary 2.4]{cuong15}.
\end{proof}

The mixed form type inequality \cite{kolodziej05}, \cite{dinew09} has been extended to the Hermitian setting in \cite{cuong15}. It will be used many times in the sequel.

\begin{lemma}[mixed form type inequality]\label{mixed-type-ineq}
Let $0 \leq f, g \in L^1(\omega^n)$ and $u, v \in PSH(\omega) \cap L^\infty(X)$. Suppose that $\omega_u^n \geq f \omega^n$ and $\omega_v^n  \geq g \omega^n$ on $X$. Then for $ k = 0, ...,  n$
\[
	\omega_u^k \wedge \omega_v^{n-k} 
		\geq f^{\frac{k}{n}} \, g^{\frac{n-k}{n}} \omega^n \quad \mbox{ on } \quad 
	X.
\]
In particular, for $0< \delta<1$,
\[
	\omega_{\delta u + (1-\delta) v}^n 
	\geq  \left[  \delta f^\frac{1}{n} 
			+ (1-\delta) g^\frac{1}{n} \right]^n \omega^n
	\quad \mbox{ on } \quad 
	X.
\]
\end{lemma}
\begin{proof}
See \cite[Lemma 1.9]{cuong15}.
\end{proof}

The uniqueness of the  Monge-Amp\`ere equation  in the realm of smooth functions is proven in \cite{TW10b}. A priori there could exist some different continuous weak solutions as constructed in \cite{KN}.
However, if the right hand side is smooth,  any continuous solution to the same right hand side  is also smooth. It follows from the following lemma.

\begin{lemma}
\label{uniqueness-smooth}
Let $ \rho \in PSH(\omega) \cap C(X)$. Assume that   
$
	\omega^n \leq \omega_\rho^n
$
as two measures. Then $\rho$ is a constant. 
\end{lemma}

\begin{proof}
Let $G\in C^\infty(X)$ be the Gauduchon function of the metric $\omega$, i.e  $dd^c (\omega^{n-1} e^G) = 0$. The mixed forms type inequality (Lemma~\ref{mixed-type-ineq})
implies that
\[
	\omega_\rho \wedge \omega^{n-1} \geq \omega^n. 
\]
Therefore, $\omega_\rho \wedge \omega^{n-1} e^G \geq \omega^n e^G$. It follows that
$dd^c \rho \wedge \omega^{n-1} e^G \geq 0$. However, by the Stokes theorem and the Gauduchon condition
\[
	\int_X dd^c \rho \wedge \omega^{n-1} e^G = 
	\int_X \rho \, dd^c (\omega^{n-1} e^G)
	=0.
\]
 Hence,
\[
	dd^c \rho \wedge \omega^{n-1} e^G = 0 \Rightarrow
	dd^c \rho \wedge \omega^{n-1} = 0.
\]
So $\rho\in PSH(\omega) \cap C(X)$ is a harmonic function on a compact
manifold and therefore it  is a constant.
\end{proof}

One of the steps to prove the stability estimate in Section~\ref{S2} 
is the construction of the barrier function having large Monge-Amp\`ere mass on some small set. The following result will help to do this. It is essentially contained in \cite{KN},  although  not exactly in this form.

\begin{proposition}
\label{barrier-funct} Fix a positive constant $A_0$ and consider functions 
 $0\leq f \in L^p(\omega^n)$, $p>1$,  such that $\int_X f \omega^n>0$ and $\|f\|_p <A_0$. 
Suppose one can solve the complex Monge-Amp\`ere equation
\[
	(\omega + dd^c w)^n = c  f \omega^n, \quad \sup_X w =0,
\]
where $w\in PSH(\omega) \cap C(X)$ and $c>0$. 
Then there exists a constant $V_{min}>0$ depending only on $X, \omega, A_0$ such that whenever
\begin{equation}
\label{mass-c1}
		\int_X f \omega^n \leq  2V_{min},
\end{equation}
we have $c\geq 2^n$. 
\end{proposition}

\begin{proof}
Suppose  $c \leq 2^n$. We shall see that this  leads to a contradiction for some positive $V_{min}$. So we  have
\[
	\omega_w^n \leq 2^n  f \omega^n.
\]
Therefore, by H\"older's inequality, for any Borel set $E\subset X$,
\[
	\int_E \omega_w^n 
\leq 		2^n \| f \|_p \left[Vol_\omega(E) \right]^\frac{1}{p*}.
\]
The volume-capacity inequality \cite[Corollary 2.4]{DK12} or \cite[Proposition 5.1]{KN} implies that the Monge-Amp\`ere measure $\omega_w^n$ satisfies the inequality  \cite[Eq. (5.2)]{KN} for the admissible (in the sense used in \cite{KN})  function 
\[
	h(x) = \frac{C e^{a x}}{2^n A_0} ,
\]
where $a>0$ is a universal number depending only on $p, X, \omega$.
Inequality \cite[Eq. (5.12)]{KN} for $0< t\leq  \frac{1}{3} \min\{\frac{1}{2^n}, \frac{1}{2^7 B}\}$ then gives:
\[
	t^n cap_\omega(\{w < S + t\}) \leq C \int_{\{\rho < S + 2t\}} \omega_w^n
	\leq C \int_X  2^n  f \omega^n ,
\]
where $S= \inf_X w$ and $C>0$ depends only on $n,B$. It implies that
\begin{equation}
\label{mc-eq2}
	\frac{t^n}{2^n C} \, cap_\omega(\{w < S + t\}) \leq \int_X  f \omega^n.
\end{equation}
Let us recall from \cite[Theorem 5.3]{KN} the formula for the function $\kappa(x)$   corresponding to $\omega_w^n$, 
\[
	\kappa(s^{-n}) = 4 C_n \left\{\frac{1}{[h(s)]^\frac{1}{n}} 
			+ \int_s^\infty \frac{dx}{x [h(x)]^\frac{1}{n}} \right\}.
\]
It is defined on $(0, cap_\omega(X))$.
Since $\kappa (x)$ is an increasing function it has the inverse $\hbar(x)$. It follows from \cite[Theorem 5.3]{KN} that for $0< t\leq  \frac{1}{3} \min\{\frac{1}{2^n}, \frac{1}{2^7 B}\}$ one has
\[
	\hbar (t) \leq cap_\omega (\{w < S +t\}). 
\]
Coupling this  with \eqref{mc-eq2} we obtain
\begin{equation}
\label{l1-apriori-bound}
	\int_X f \omega^n \geq \frac{t^n \hbar(t)}{2^n C}.
\end{equation}
Define 
\begin{equation}
\label{v-min}
	V_{min}:= \frac{t_0^n}{2^{n+2} C\hbar(t_0)}>0, \quad 
	t_0 = \frac{1}{6} \min\{\frac{1}{2^n}, \frac{1}{2^7 B}\}.	
\end{equation}
Then, the inequality \eqref{l1-apriori-bound} and the assumptions would give a contradiction
\[
	2 V_{min} \geq \int_X f \omega^n \geq 4 V_{min}>0.
\]
Thus the proposition is proven.
\end{proof}

The following minimum principle is inspired by \cite[Theorem A]{BT76}, which was stated for a bounded open set in $\co$. 
In the Hermitian setting it requires a different proof.

\begin{proposition}[minimum principle]
\label{min-prin}
Let $U\subset \subset X$ be a non-empty open set. Let $u, v\in PSH(\omega) \cap C(X)$ be such that 
\[
	  c\, \omega_u^n \leq \omega_v^n \quad \mbox{on } U
\]
for some $c>1$. Then, 
\[	
	\min\{u(x)-v(x): x \in \overline{U}\} = \min\{u(x) - v(x): x\in \partial U\}.
\]
\end{proposition}

\begin{proof}
Without loss of generality suppose that $\min\{u(x) - v(x): x\in \partial U\} = 0$, i.e.
\[
	v \leq u \quad \mbox{on } \partial U.
\]
We need to show that $\min\{u(x)-v(x): x \in \overline{U}\} =0$. Suppose that it is not the case. Then there exists $x_0\in U$ such that $u(x_0) < v(x_0)$. So,
\[
	S := \min\{u(x) - v(x): x \in U\} <0. 
\]
Let  us use the following  notation 
\[
	S_\varepsilon := \min_{\overline U}\{u - (1-\varepsilon) v\},
\]
for $0< \varepsilon <<1$.
Then
\[
	S - \varepsilon \|v\|_\infty \leq S_\varepsilon \leq S +  \varepsilon \|v\|_\infty .
\]
Therefore, 
\[
	U(\varepsilon,t):= \{u< (1-\varepsilon) v + S_\varepsilon + t \} \cap U \subset 
	\{u < v + S + 2 \varepsilon \|v\|_\infty+ t\} \cap U \subset\subset U ,
\]
for every $0< t, \varepsilon<\delta_0$ with $\delta_0$ small enough,  as $u\geq v$ on $\partial U$. 

From
the modified comparison principle \cite[Theorem 0.2]{KN} we have
\begin{equation}
\label{min-prin-eq8}
	0< \int_{U(\varepsilon,t)} \omega_{(1-\varepsilon)v}^n \leq [1+ \frac{C t}{\varepsilon^n}] \int_{U(\varepsilon,t)} \omega_u^n ,
\end{equation}
for every $0< t < \frac{\varepsilon^3}{16B}$, where $0< \varepsilon<\delta_0$.
Since 
\begin{equation}
\label{min-prin-eq9}
	\omega_{(1-\varepsilon)v}^n\geq (1-\varepsilon)^n\omega_v^n \geq (1-\varepsilon)^n c\, \omega_u^n,
\end{equation}
and $c>1$, we can choose $0<\varepsilon<\delta_0$ so that 
\begin{equation}
\label{min-prin-eq10}
	(1-\varepsilon)^n c >1.
\end{equation}
It follows from \eqref{min-prin-eq8}, \eqref{min-prin-eq9} and \eqref{min-prin-eq10} that 
\[
	[(1-\varepsilon)^nc -1] \int_{U(\varepsilon,t)} \omega_u^n
	\leq \frac{Ct}{\varepsilon^n} \int_{U(\varepsilon,t)} \omega_u^n.
\]
Therefore, for fixed $0< \varepsilon< \delta_0$ and for every $0< t< \frac{\varepsilon^3}{16B}$
\[	
	[(1-\varepsilon)^nc -1] < \frac{Ct}{\varepsilon^n}.
\]
It is impossible. The proof is completed.
\end{proof}

We finish this section with a lemma   from Riemannian geometry, which follows from Sobolev's embedding theorem.

\begin{lemma}
\label{bbg}
Let $(X,\omega)$ be a compact $n$-dimensional Hermitian manifold. 
Let $\psi \in C(X) \cap W^{1,2}(X)$ be a real-valued function. Fix $d, \delta >0$. Assume that
\[
	Vol_\omega(\{\psi < 0\}) \geq \delta, \quad
	Vol_\omega(\{\psi \geq d\}) \geq \delta .
\]
The square of the norm of the gradient is given in local coordinates by
\[
	|\partial \psi|^2 = \sum g^{k\bar l} \partial_k \psi \partial_{\bar l} \psi, \ \ \text{where} \ \ 
	\quad \omega = \frac{i}{2} \sum g_{k\bar l} dz_k \wedge d\bar z_l .
\]
Then,
\[
	\int_{\{0 < \psi < d\}} |\partial \psi| \omega^n 
	\geq C \, d \, \delta^{\frac{4n-1}{2n}} ,
\]
where $C >0$ is a uniform constant depending only on $X, \omega$. 
\end{lemma}

\begin{proof} We apply Sobolev's embedding theorem to the function
$$
\Psi = \max (\psi , 0) - \max (\psi -d , 0)  - M,
$$
where $M= \int _X [\max (\psi , 0) - \max (\psi -d , 0)] \om ^n$. It gives
\[
	\|\Psi\|_{\frac{2n}{2n-1}} \leq C \int_X |\partial \Psi | \omega^n.
\]
It is easy to see that 
\[
	\int_X |\partial \Psi | \omega^n 
	= \int_{\{0 < \psi < d\}} |\partial \psi | \omega^n.
\]
As $Vol_\omega(\{\psi \geq d\}) \geq \delta$, we have
\[
	M = \int_{\{0< \psi < d\}} \psi \omega^n + d \int_{\{\psi \geq d\}} \omega^n
	\geq d \, \delta.
\]
Moreover, as $Vol_\omega(\{\psi <0\})\geq \delta$,
\begin{align*}
&	\int_X |\max(\psi,0) - \max(\psi -d, 0) - M|^\frac{2n}{2n-1} \omega^n\\	
&\geq 		\int_{\{\psi \leq 0\}} M^\frac{2n}{2n-1} \omega^n
\geq 		\int_{\{\psi \leq 0\}} (d \, \delta)^\frac{2n}{2n-1} \omega^n \\
&\geq	d^\frac{2n}{2n-1} \delta^{1 + \frac{2n}{2n-1}}.
\end{align*}
Therefore, we conclude that
$
	\int_{\{0 < \psi < d\}} |\partial \psi| \omega^n 
	\geq C \, d \, \delta^{1+ \frac{2n-1}{2n}}.
$
\end{proof}

\section{stability estimates and uniqueness} 
\label{S2}

In this section we  prove our main result: the stability and uniqueness of the Monge-Amp\`ere equation with the positive right hand side in $L^p(\omega^n)$, $p>1$.

\begin{theorem}
\label{stability-smooth-lp}
Let $0\leq f, g\in L^p(\omega^n)$, $p>1$, be such that $\int_X f \omega^n >0$, $\int_X g\omega^n>0$. Consider two continuous $\omega$-psh solutions of the complex Monge-Amp\`ere equation 
\[
	\omega_u^n = f\omega^n, \quad 
	\omega_v^n = g\omega^n
\] 
with 
$
	\sup_X u = \sup_X v =0.
$
Assume that $f$ is  smooth and 
\[
	 f \geq c_0>0. 
\] 
Fix $0< \alpha < \frac{1}{n+1}$. Then, there exists $C = C(c_0, \alpha, \|f\|_p, \|g\|_p)$ such that
\[
	\|u-v\|_\infty \leq C \|f-g\|_p^\alpha.
\]
\end{theorem}


\begin{observation} We shall remove the smoothness assumption on $f$  thus obtaining Theorem A from Introduction (see Remark~\ref{relax-smooth}).
\end{observation}
\begin{observation}
\label{ob-us} It follows from Lemma~\ref{uniqueness-smooth} that  the continuous $\omega$-psh function $u$ is automatically smooth as it coincides with the unique (normalised) smooth solution obtained by Tosatti and Weinkove \cite{TW10b}. 
\end{observation}

\begin{proof}

We shall use the notation: 
\[
	 \varphi = u-v , \quad  t_0 = \min_X \varphi , \quad  \Omega (t)=\{ \varphi <t \}, \quad
	   \| f-g \|_p  \leq \varepsilon.
\]
Assuming $0<\varepsilon <<1$  we wish to show that
$\| \varphi \|_{\infty } \leq C\varepsilon^\alpha$ with a fixed $0< \alpha<\frac{1}{n+1}$. 
Since $Vol_\omega(X) =1$, it follows that $\|f-g\|_r \leq \|f-g\|_p \leq \varepsilon$ and
\[
	\int_E |f-g| \omega^n \leq \left(\int_E |f-g|^r \omega^n\right)^\frac{1}{r} 
	\leq \left(\int_E |f-g|^p \omega^n\right)^\frac{1}{p} \leq \varepsilon , 
\]
for every $1\leq r\leq p$ and every Borel set $E\subset X$.





\begin{lemma}
\label{case-1} 
Let $V_{min}>0$ be the constant from Proposition~\ref{barrier-funct} with $A_0 = 2\|f\|_p$. Fix $t_1 > t_0$.
If $\int_{\Omega(t_1)} f \omega^n \leq V_{min}$, then 
\begin{equation}
\label{barrier-case}
	t_1 - t_0 \leq C \varepsilon^\alpha ,
\end{equation}
where $0< \alpha < \frac{1}{n+1}$ is fixed.
\end{lemma}

\begin{proof}
Consider the following sets 
\[
	\Omega_1 = \{z\in \Omega(t_1): f(z) \leq (1+\varepsilon^\alpha) g(z)\}, \quad
	\Omega_2 = \Omega(t_1) \setminus \Omega_1.
\]
We have
\begin{equation}
\begin{aligned}
	\int_{\Omega_2} f\omega^n 
&\leq 	\int_{\Omega_2} |f-g| \omega^n + \int_{\Omega_2} g \omega^n\\
&\leq		\varepsilon + \int_{\Omega_2} \varepsilon^{-\alpha} (f-g)  \omega^n \\
&\leq		\varepsilon + \varepsilon^{1-\alpha} \leq 2\varepsilon^{1-\alpha}. 
\end{aligned}
\end{equation}
Moreover, 
\begin{equation} 
\label{?}
\begin{aligned}
	\int_{\Omega_2} f^p\omega^n 
&\leq 	\int_{\Omega_2} 2^{p-1} (|f - g|^p + g^p) \omega^n \\
&\leq		\int_{\Omega_2} 2^{p-1} |f-g|^p \omega^n 
			+ \int_{\Omega_2} 2^{p-1}[\varepsilon^{-\alpha} (f-g)]^p \omega^n \\
&\leq		2^{p-1}(\|f-g\|_p^p + \|f-g\|_p^p \varepsilon^{-p\alpha}) \\
&\leq		2^{p-1}(\varepsilon^p+ \varepsilon^{p- p\alpha}) 
\leq 		2^p \varepsilon^{p-p\alpha}.
\end{aligned}
\end{equation}
Define for $0< \varepsilon <<1$ and $A>>1$ (to be chosen later)
\[
	\hat f(z) =	\begin{cases}
		f(z)  \quad &\mbox{for } z\in \Omega_1, \\
		\varepsilon^{-n\alpha} f(z) \quad &\mbox{for } z\in \Omega_2, \\
		\frac{1}{A} f(z) \quad	&\mbox{for } z\in X \setminus \Omega(t_1).
	\end{cases}
\]
Then
\begin{equation}
\label{f-hat-l1}
\begin{aligned}
	\int_X \hat f \omega^n
&=		\int_{\Omega_1} f \omega^n 
		+ \int_{\Omega_2} \varepsilon^{-n\alpha} f \omega^n
		+ \int_{X\setminus \Omega(t_1)} \frac{1}{A} f \omega^n \\
&\leq		V_{min} + 2\varepsilon^{1- \alpha - n\alpha}+ \frac{1}{A}\|f\|_1 ,
\end{aligned}
\end{equation}
and
\begin{equation}
\label{f-hat-lp}
\begin{aligned}
		\int_X \hat f^p \omega^n
&=		\int_{\Omega_1} f^p \omega^n
		+ \int_{\Omega_2} \varepsilon^{-np\alpha} f^p \omega^n
		+ \int_{X\setminus \Omega(t_1)} \frac{1}{A^p} f^p \omega^n \\
&\leq		\|f\|_p^p + 2^p \varepsilon^{p-p\alpha - n p \alpha}
			+ \frac{1}{A^p}\|f\|_p^p. \\
\end{aligned}
\end{equation}
Since $0< \alpha < \frac{1}{n+1}$,  
 we can choose $0< \varepsilon <<1$ and $A>>1$ such that
\begin{equation}
\label{epsilon-cond}
	\max\{ 2^p \varepsilon^{1 - (n+1)\alpha},  \frac{1}{A} \|f\|_p \} 
\leq 		\min\{\frac{V_{min}}{4}, \frac{\|f\|_1}{2}\}.
\end{equation}
Notice that by our normalisation $Vol_\omega(X) =1$, so we have
\[
	\|f\|_r \leq \|f\|_p \quad \mbox{for } 1\leq r \leq p.
\]
It implies that for such $\varepsilon$ and $A$,
\begin{equation}
\label{f-ha-l1-lp}
	\int_X \hat f \omega^n \leq \frac{3}{2} V_{min} \quad \mbox{and} \quad
	\|\hat f\|_p \leq 2 \|f\|_p .
\end{equation}
By \cite[Theorem 0.1]{KN} there exists $w \in PSH(\omega) \cap C(X)$ and $\hat c>0$ solving
\begin{equation}
\label{c1-bf-w}
	\omega_w^n = \hat c \hat f \omega^n, \quad \sup_X w =0.
\end{equation}
It follows from Proposition~\ref{barrier-funct} and \eqref{f-ha-l1-lp} that 
\begin{equation}
\label{c1-bf-w-c}
	\hat c \geq 2^n.
\end{equation}
Moreover, as $\hat f \geq f/A$, using Lemma~\ref{const-comp},  we get that
\[
	\hat c \leq A.
\] 
Let us define for $0<s<1$,
$
	\psi_s = (1-s)v + s w .
$
By the mixed form  type inequality (Lemma~\ref{mixed-type-ineq})  we have
\begin{align*}
	\omega_{\psi_s}^n \geq [(1-s) g^\frac{1}{n} + s (\hat c\hat f)^\frac{1}{n} ]^n \omega^n
&	= 	\left[(1-s)\left(g/f\right)^\frac{1}{n} + s \left(\hat c\hat f/f\right)^\frac{1}{n}\right]^n  f \omega^n \\
&	=: [b(s)]^n f \omega^n. 
\end{align*}
We shall see that $b(s)>1$ for $2\varepsilon^\alpha \leq s \leq 1/2 $ on $\Omega(t_1)$. Indeed,  for  $z\in \Omega_1$, we have $g/f \geq 1/(1+\varepsilon^\alpha)$ and $\hat c\hat f \geq 2^n f$.
Therefore, on $\Omega_1$,
\begin{equation}
\label{bc-eq15}
	b(s) \geq \frac{1-s}{(1+\varepsilon^\alpha)^\frac{1}{n}} + 2s 
	\geq \frac{1-s}{1+\varepsilon^\alpha} + 2s.
\end{equation}
Now, for $z\in \Omega_2$, we have $\hat c \hat f \geq 2^n \varepsilon^{-n\alpha}f$. Therefore, on $\Omega_2$
\begin{equation}
\label{bc-eq16}
	b(s) \geq 2s \varepsilon^{-\alpha}.
\end{equation}
It follows from \eqref{bc-eq15} and \eqref{bc-eq16} that if 
\[
	2 \varepsilon^\alpha \leq s \leq \frac{1}{2},
\]
then 
\[
	b(s) > 1+ \varepsilon^\alpha
\]
on $\Omega(t_1) = \Omega_1 \cup \Omega_2$. This means that for such a value of $s$, 
\[ 
	\omega_{\psi_s}^n > (1+\varepsilon^\alpha) f \omega^n \quad
	 \mbox{on } \Omega(t_1).
\]

Applying the minimum principle (Proposition~\ref{min-prin}) for $s= 2 \varepsilon^\alpha$ and the function 
\[
	u - [(1-s)v + s w] = \varphi - s(w-v)
\]
on $\Omega(t_1)$, we get that
\[
	\min_{\overline{\Omega(t_1)}} [ \varphi(z) - s (w-v)(z) ] = 
	\min_{\partial\Omega(t_1)}  [\varphi(z) - s (w-v)(z)].
\]
Therefore
\[
	t_1 - t_0 \leq 2 s \|w-v\|_\infty 
\leq 	4 (\|v\|_\infty + \|w\|_\infty) \varepsilon^\alpha . 
\]
Since $\hat c \leq A$, by \cite[Corollary 5.6]{KN}, we have 
\[	
	\|w\|_\infty \leq  A^\frac{1}{n} H,
\]
where $H = H(\|f\|_p, \|g\|_p, X, \omega)>0$ is a uniform bound for $u,v$, i.e.
\[
	- H \leq u, v \leq 0.
\]
Thus  Lemma~\ref{case-1} follows.
\end{proof}

\bigskip

We pass to the second part of the proof. 
The main ingredient here are   {\em a priori} estimates for the Laplacian $\omega_u \wedge \omega^{n-1}$ in collars $\{a< \varphi < b\}$. 
Recall our notation:
\[
	0< c_0 \leq f \in C^\infty(X),
\]
and 
\[	
	\omega_u^n = f\omega^n, \quad \omega_v^n = g\omega^n,
\]
where 
\[
	 u\in C^\infty(X), \quad \omega + dd^c u>0, \quad v\in PSH(\omega) \cap C(X).
\]
Moreover,
\[
	t_0 = \min_X \varphi, \quad \|f-g\|_p \leq \varepsilon.
\]
\begin{lemma}[Laplacian mass on small collars]
\label{lap-est-collar}
Let $t_0 < a<b$ be two real numbers satisfying 
\[
	\varepsilon \leq b-a \quad \mbox{and}\quad  a-t_0 = b- a .
\] 
Assume that 
\begin{equation}
\label{bbg-iso-ass}
	Vol_\omega (\{\varphi < a\}) \geq \delta \quad
	\mbox{and} \quad
	Vol_\omega(\{\varphi \geq b\}) \geq \delta,
\end{equation}
for some $\delta>0$.
Then,
\[
	\int_{\{a < \varphi < b\}} \omega_u \wedge \omega^{n-1} \geq \delta_e>0,
\]
where $\delta_e = C c_0 \,\delta^\frac{4n-1}{n}> 0$  (thus depends only on  $\delta, c_0, X, \omega$).
\end{lemma}

\begin{proof} 
We first estimate
\begin{equation}
\label{lec-eq1}
	\int_{\{a < \varphi < b\}}d\varphi \wedge d^c \varphi \wedge \omega_u^{n-1} 
\end{equation}
from above and then 
\begin{equation}
\label{lec-eq2}
	\int_{\{a < \varphi < b\}} \omega_u \wedge \omega^{n-1} \cdot
		\int_{\{a < \varphi < b\}}d\varphi \wedge d^c \varphi \wedge \omega_u^{n-1} 
\end{equation}
from below.
Let us start with the first one.
\begin{equation}
\label{lec-eq3}
	\int_{\{a < \varphi < b\}}d\varphi \wedge d^c \varphi \wedge \omega_u^{n-1} 
	\leq \int_{\{a < \varphi < b\}}d\varphi \wedge d^c \varphi \wedge T ,
\end{equation}
where 
\[
	T = \sum_{k=0}^{n-1} \omega_u^k \wedge \omega_v^{n-1-k} \quad
	\mbox{and}\quad
	\omega_u^n - \omega_v^n = dd^c \varphi \wedge T.
\]
Then
\[
	 dd^c T  = dd^c\omega \wedge T_1 + d\omega \wedge d^c \omega\wedge T_2
\]
where $T_1, T_2$ are positive currents (see \cite{DK12}).
Therefore, by the choice  of the constant $B>0$  (see \eqref{curv-const})  
\[
	- B (\omega^2 \wedge T_1 + \omega^3\wedge T_2) \leq dd^c T
\leq  	B (\omega^2 \wedge T_1 + \omega^3\wedge T_2) .
\]
It follows that for any Borel set $E\subset X$ and a continuous function $w \geq 0$ on $X$ we have (see e.g. \cite[Proposition 1.5]{cuong15})
\begin{equation}
\label{lec-cln}
\begin{aligned}
	\left| \int_E w dd^c T \right| 
&\leq 	B \|w\|_{L^\infty(E)} \int_X (\omega^2 \wedge T_1 + \omega^3\wedge T_2) \\
&\leq 	C B \|w\|_{L^\infty(E)} (1+\|u\|_\infty)^n (1+\|v\|_\infty)^n.
\end{aligned}
\end{equation}
To simplify notation, we write in what follows
\[
	\psi:= \varphi -a, \quad \mbox{(and then} \quad \min_X \psi= t_0 - a ).
\]
Notice that this does not affect  $T$ and we still have $dd^c \psi \wedge T = \omega_u^n -\omega_v^n$. The right hand side in the inequality \eqref{lec-eq3} becomes
\begin{equation}
\label{lec-eq2-change}
	\int_{\{0< \psi < b-a\}}d\psi \wedge d^c \psi \wedge T .
\end{equation}
By the Stokes theorem,
\begin{align*}
&	\int_{\{0< \psi < b-a\}}d\psi \wedge d^c \psi \wedge T \\
&	=	\int_{\{0 < \psi < b-a\}}  d(\psi d^c \psi \wedge T)  
		- \int_{\{0 < \psi < b-a\}} \psi dd^c \psi \wedge T 
		 +  \int_{\{0 < \psi < b-a\}} \psi d^c \psi \wedge d T \\
&	= 	\int_{\partial \{0 < \psi < b-a\}} \psi d^c \psi \wedge T 
		- \int_{\{0 < \psi < b-a\}} \psi (f-g)\omega^n 
		- \int_{\{0 < \psi < b-a\}} \psi d \psi \wedge d^c T \\
&	=	(b-a) \int_{\partial \{0 < \psi < b-a\}} d^c \psi \wedge T 
		- \int_{\{0 < \psi < b-a\}} \psi (f-g)\omega^n 
		- \int_{\{0 < \psi < b-a\}}\psi d \psi \wedge d^c T \\
&	=	(b-a) \left(\int_{\{\psi<b-a\}} dd^c \psi \wedge T 
		- \int_{\{\psi<b-a\}}d^c\psi \wedge dT\right)  \\
&	\qquad\qquad	- \int_{\{0 < \psi < b-a\}} \psi d \psi \wedge d^cT
			- \int_{\{0 < \psi < b-a\}} \psi (f-g)\omega^n\\
&	:= (b-a)(J_1 + J_2) - J_3 - J_4,
\end{align*}
where we used $dd^c \psi \wedge T = (f-g)\omega^n$ for the second inequality and the Stokes theorem once more for the fourth equality. A few words of explanation how this theorem is used when the boundary of the set $\{0<\psi< b-a\}$ is not smooth. One needs to use the approximation argument. First, we assume that $v$ is also smooth, then by Sard's theorem and the Lebesgue domination theorem we process the proof as above. Next, choosing a smooth sequence $\omega $-psh functions $\{v_j\}$ decreasing (uniformly) to $v$ we  get the equality of the first integral and the last sum corresponding to $v_j$. Finally, pass to the limit by using Bedford-Taylor's convergence theorem in the Hermitian setting \cite{DK12} to get equality in the above estimates. We refer the reader to \cite[Theorem~4.1]{BT82} or \cite[Theorem~1.16]{kolodziej05} for more details of this argument in the case $v$ is only bounded, where the quasi-continuity of plurisubharmonic function is essential. Below we shall use the Stokes theorem several times in this fashion.

Our next step is to bound $|J_k|$, $k =1,2,3,4$. However, it is not hard to see from the computation below that one can deal with $J_k$ and $ -J_k,$  in the same way. So we only give the estimates from above for $J_k$.

The  bounds for $J_1, J_4$  are easy:
\begin{equation}
\label{j-1}
	J_1 = \int_{\{\psi<b-a\}} (\omega_u^n - \omega_v^n) 
		\leq   \int_X |f-g| \omega^n 
		\leq \varepsilon ,
\end{equation}
and
\begin{equation}
\label{j-4}
	J_4  	= \int_{\{0 < \psi < b-a\}} \psi (f-g)\omega^n \leq 
	\int_{\{0 < \psi < b-a\}} \psi |f-g|\omega^n  \leq
	(b-a) \varepsilon.
\end{equation}
To estimate $J_2, J_3$ we will use the inequality \eqref{lec-cln} several times without  mentioning it anymore. First, we consider  $J_2$. Set
\[
	J_2' = - J_2 - \int_{\{\psi \leq 0\}} d\psi \wedge d^c T .
\]
Then, using the Stokes theorem twice, we get that
\begin{align*}
	J_2' 
&	=	 \int_{\{0< \psi <b-a\}} d \psi \wedge d^cT \\
&	= 	\int_{\{0< \psi <b-a\}} d (\psi d^c T) -
		\int_{\{0< \psi <b-a\}} \psi dd^c T \\
&	=	(b-a) \int_{\partial \{0 < \psi < b-a\}}  d^c T  - 
		\int_{\{0<\psi <b-a\}} \psi dd^cT \\
&	=	(b-a) \int_{\{\psi< b-a\}} dd^c T - 
		\int_{\{0<\psi <b-a\}} \psi dd^cT \\
&	\leq	C (b-a).
\end{align*}
Similarly, by the Stokes theorem,
\[
	\int_{\{\psi \leq 0\}} d\psi \wedge d^c T  
	= - \int_{\{\psi \leq 0\}} \psi dd^c T 
	\leq C (a- t_0) = C(b-a) ,
\]
where we used the fact that $\min_X \psi = t_0 - a = -(b-a)$.
This implies that
\begin{equation}
\label{j-2}
	|J_2| \leq C(b-a) .
\end{equation}
Again, using the Stokes theorem twice, we get that
\begin{align*}
	2 J_3
&	= 	\int_{\{0 < \psi < b-a\}} d \psi^2 \wedge d^c T  \\
&	= 	\int_{\{0 < \psi < b-a\}} d (\psi^2 d^c T) - 
		\int_{\{0 < \psi < b-a\}} \psi^2 dd^c T  \\
&	= 	(b-a)^2 \int_{\partial \{0 < \psi < b-a\}} d^c T -
	 \int_{\{0 < \psi < b-a\}} \psi^2 dd^c T \\
&	=	(b-a)^2 \int_{\{\psi < b-a\}} dd^c T  - 
	 \int_{\{0 < \psi < b-a\}} \psi^2 dd^c T .	
\end{align*}
It  follows that
\begin{equation}
\label{j-3}
	|J_3| \leq C (b-a)^2.
\end{equation}
Combining the inequalities \eqref{lec-eq3}, \eqref{j-1}, \eqref{j-4},  \eqref{j-2} and  \eqref{j-3} one obtains
\begin{equation}
\label{lec-eq14}
\begin{aligned}
	\int_{\{0< \psi < b-a\}}d\psi \wedge d^c \psi \wedge \omega_u^{n-1} 
&	\leq 	C \left[2 \varepsilon(b-a) + (b-a)^2 + (b-a)^2  \right] \\
&	\leq 4C (b-a)^2 .
\end{aligned}
\end{equation}
This is the desired  upper bound. 
It remains to estimate from below the quantity
\[
	I_u(a,b):= \int_{\{0< \psi < b-a\}} \omega_u \wedge \omega^{n-1} \cdot
		\int_{\{0 < \psi < b-a\}}d\psi \wedge d^c \psi \wedge \omega_u^{n-1}. 
\]
Since $u, v\in PSH(\omega) \cap C(X) \subset W^{1,2}(X)$, 
we have $\psi \in W^{1,2}(X)$. Applying Lemma~\ref{bbg} for $\psi$, $d=b-a$ and $\delta$ we get that
\begin{equation}
\label{lec-eq17}
	\int_{\{0< \psi < b-a\}} |\partial \psi|\omega^n 
\geq 	C (b-a) \delta^\frac{4n-1}{2n}.
\end{equation}

\begin{lemma}
\label{led-lower-bound} 
We have 
\[
I_u(a,b) \geq \frac{c_0}{n^2}\left(\int_{\{0 < \psi < b-a\}} |\partial \psi| \omega^n\right)^2
\]
where 
\[
	|\partial \psi|^2 = \sum g^{k\bar l} \partial_k \psi \partial_{\bar l} \psi,
	\quad \omega = \frac{i}{2} \sum g_{k\bar l} dz_k \wedge d\bar z_l
\]
in a local coordinate chart.
\end{lemma}

\begin{proof} Recall that $f>0$ is smooth, hence \cite{TW10b} implies that $u$ is smooth.
As $\psi \in C(X)\cap W^{1,2}(X)$, there exist $\psi_j\in C^\infty(X)$ such that $\psi_j \to \psi$ uniformly and  $\psi_j \to \psi$ in $W^{1,2}(X)$. Therefore, 
upon applying the convergence theorem \cite{BT82}, \cite{DK12},
we may assume $\psi$ is smooth. Then, the inequality is a consequence of the Cauchy-Schwarz inequality and the elementary point-wise inequality
\begin{align*}
	\frac{\omega_u \wedge \omega^{n-1}}{\omega^n} \cdot \frac{\omega_u^{n-1} \wedge d\psi \wedge d^c \psi}{\omega^n}  
&	\geq		\frac{1}{n} \frac{\omega_u^n}{\omega^n} \cdot \frac{d \psi \wedge d^c \psi \wedge \omega^{n-1}}{\omega^n} \\
&	\geq  \frac{c_0}{n^2} |\partial \psi|^2.
\end{align*}
The first inequality can be checked by writing it down in normal  coordinates at any given point. 
This is the sole place where we need to use the assumption on smoothness of $f$ and $u$. The approximation process would not work for the proof of the lemma because up to this point  the uniqueness of continuous solutions is not yet asserted. 
\end{proof}

Thanks to \eqref{lec-eq17} and Lemma ~\ref{led-lower-bound} we get that
\begin{equation}
\label{lec-eq31}
I_u(a,b)
\geq	C c_0 \, (b-a)^2 \delta^\frac{4n-1}{n}.
\end{equation}
This is the lower bound we needed. 
Combining the upper bound \eqref{lec-eq14} and the lower bound \eqref{lec-eq31} we obtain that
\begin{align*}
	\int_{\{0< \psi < b-a\}} \omega_u \wedge \omega^{n-1}
&	\geq		\frac{C c_0  (b-a)^2 \delta^\frac{4n-1}{n}}{4 C (b-a)^2} \\
&	\geq 		C c_0 \, \delta^\frac{4n-1}{n} 
:= 	\delta_e>0.
\end{align*}
Going back to the original notation  $\psi= \varphi+a$ we get the desired inequality
\[
	\int_{\{a< \varphi < b\}} \omega_u \wedge \omega^{n-1}
	\geq \delta_e.
\]
The lemma is proven.
\end{proof}

Let us observe that from   Lemma~\ref{lap-est-collar} and its proof, by the symmetry 
  with respect to $u$ and $v$, one obtains also the following statement.

\begin{observation}
\label{Laplacian-mass-2} Let $\hat t_0 = \max_X\varphi$. Let $a < b< \hat t_0$ be two real numbers such that
\[
	\varepsilon \leq b - a, \quad \mbox{and} \quad 
	\hat t_0 - b = b - a.
\]
Assume that 
\[
	Vol_\omega(\{\varphi < a\}) \geq \delta, \quad
	Vol_\omega(\{\varphi \geq b \}) \geq \delta
\]
for some $\delta>0$. Then, 
\[
	\int_{\{a < \varphi < b\}} \omega_u \wedge \omega^{n-1} \geq \delta_e>0
\]
where $\delta_e$ as in Theorem~\ref{lap-est-collar} depends only on $\delta, c_0, X, \omega$.
\end{observation}

Thanks to {\em a priori} estimates of the Laplacian mass on small collars we get the following estimate for a larger collar (there is a similar statement corresponding to Remark~\ref{Laplacian-mass-2}). 

\begin{proposition}
\label{lap-est} 
Assume that $t_1-t_0\geq \varepsilon$ and $Vol_\omega(\Omega(t_1))\geq \delta$. Define for $k\geq 2$ integer,
\[
	t_k = 2^{k-1} (t_1 - t_0) + t_0.
\]
If $Vol_\omega(X\setminus \Omega(t_N)) \geq \delta$, $N\geq 1$, then we have
\[
	\int_{\{t_0 < \varphi \leq t_N\}} \omega_u \wedge \omega^{n-1} 
	\geq (N-1) \delta_e ,
\]
where $\delta_e =  C c_0 \, \delta^\frac{4n-1}{n}$.
\end{proposition}

\begin{proof}
Observe that $t_k - t_{k-1} = t_{k-1}-t_0 \geq \varepsilon$ for $k\geq 2$. The assumptions of Lemma~\ref{lap-est-collar} are satisfied for $a= t_{k-1}, b= t_k$, $k =2, ..., N$. Then, we get that
\[
	\int_{\{t_0 < \varphi \leq t_N\}} \omega_u \wedge \omega^{n-1} 
\geq		\sum_{k=2}^N \int_{\{t_{k-1}< \varphi \leq t_k\}} 
			\omega_u \wedge \omega^{n-1} 
\geq 		(N-1) \delta_e
\]
where $\delta_e = Cc_0 \, \delta^\frac{4n-1}{n}$.
\end{proof}

We have completed estimates on the level sets near the minimum and the maximum of the difference $\varphi= u-v$. Moreover, Lemma~\ref{lap-est-collar} gives bounds for the Laplacian mass of remaining level sets on the manifold. Now by combining them we will finish the proof of the stability theorem.

\medskip

{\em End of proof of Theorem~\ref{stability-smooth-lp}.} 
We have that
\[
	t_0 = \min_X \varphi, \quad \|f-g\|_p \leq \varepsilon.
\]
Set
\[
	- \hat t_0 = \min_X (-\varphi) = - \max_X \varphi
\]
Since $\sup_Xu = \sup_X v =0$, we have 
\begin{equation}
\label{end-stab-eq16}
	\hat t_0 = \max_X \varphi \geq 0, \quad 
	t_0 = \min_X\varphi \leq 0.
\end{equation}
So
\begin{equation}
\label{pr-s-s-lp-eq4}
	\|\varphi\|_\infty \leq  \hat t_0 - t_0.
\end{equation}
Our goal is to prove $\|\varphi\|_\infty \lesssim \varepsilon^\alpha$.
We consider two possibilities: 

\noindent
{\bf Case 1.}
\begin{equation}
\label{f-g-small-volume-1}
\min \left\{	\int_{\{\varphi < t_0 + \varepsilon\}} f\omega^n, 
	\quad \int_{\{\varphi > \hat t_0 - \varepsilon\}} f \omega^n \right \} 
\geq 	\frac{V_{min}}{4}.
\end{equation}
\noindent
{\bf Case 2.}
\begin{equation}
\label{f-g-small-volume-2}
\min \left\{	\int_{\{\varphi < t_0 + \varepsilon\}} f\omega^n, 
	\quad \int_{\{\varphi > \hat t_0 - \varepsilon\}} f \omega^n \right \} 
< 		\frac{V_{min}}{4}.
\end{equation}

In  the first case  it follows from \eqref{f-g-small-volume-1} and H\"older's inequality that 
\begin{equation}
\label{pr-s-s-lp-eq6}
	\min\{ Vol_\omega(\{\varphi < t_0 + \varepsilon\}), 
			Vol_\omega(\{\varphi > \hat t_0 - \varepsilon\})\} 
	\geq \delta_1  ,
\end{equation}
where 
\[
	\delta_1 = \frac{(V_{min}/4)^{p^*}}{\|f\|_p^{p^*}}>0.
\]
Define $t_N = 2^{N-1} \varepsilon + t_0$ for an integer $N\geq 1$. By \eqref{pr-s-s-lp-eq4}, if  $N$ satisfies
$	\varepsilon(2^{N-1} + 1)< \|\varphi\|_\infty ,
$ 
then
\[
	t_N < \hat t_0 - \varepsilon.
\]
Hence, applying  Proposition~\ref{lap-est},  for $t_1= t_0+\varepsilon$ and $t_N$ as above,  we get that
\begin{equation}
\label{pr-s-s-lp-eq8}
	\int_{\{t_0 < \varphi \leq t_N\}} \omega_u \wedge \omega^{n-1}
\geq		(N-1) \delta_{1e} ,
\end{equation}
where $\delta_{1e} = C c_0 \, \delta_1^\frac{4n-1}{n}$. 
Let $G\geq 0$  be the Gauduchon function of $\omega$, i.e. 
$dd^c (e^G\omega^{n-1}) =0$. By the Stokes theorem
\begin{equation}
\label{pr-s-s-lp-eq9}
	\int_X \omega_u \wedge \omega^{n-1} \leq \int_X e^G \omega_u \wedge \omega^{n-1} = \int_X e^G \omega^n.
\end{equation}
Choose \begin{equation}\label{eps-cond}
    N_1 = \frac{\int_X e^G \omega^n}{\delta_{1e}} +2. 
\end{equation}
Then, 
by \eqref{pr-s-s-lp-eq8} and \eqref{pr-s-s-lp-eq9} we have
\[
	\int_X e^G \omega^n > ([N_1] -1)  \delta_{1e} > \int_X e^G \omega^n ,
\]
where $[N_1]$ is the integer part of $N_1$. This is not possible, so for such a choice of $N_1$

\begin{equation}
\label{extra-case-1}
	\|\varphi\|_\infty \leq \varepsilon (1+ 2^{N_1}).
\end{equation}
Thus, in the first case  the statement follows.

We turn now to the second case and assume \eqref{f-g-small-volume-2}. By the symmetry of the estimate for the Laplacian mass in large collars (Remark~\ref{Laplacian-mass-2}), without loss of generality, we may assume that 
\begin{equation}
\label{f-small-mass}
	\int_{\{\varphi < t_0 + \varepsilon\}} f \omega^n \leq V_{min}/4.
\end{equation}
Choose $t_1$ to be the supremum over $t$ for which
\[
	\int_{\Omega(t)} f \omega^n < V_{min}.
\]
By \eqref{f-small-mass} and Lemma~\ref{case-1} we have
\begin{equation}
\label{end-stab-eq3}
	\varepsilon \leq t_1 - t_0 \leq C \varepsilon^\alpha.
\end{equation}
Choose $\hat t_1$ to be infimum over $t$ for which
\[
	\int_{X\setminus \Omega(t)} g \omega^n \leq V_{min} .
\]
By Lemma~\ref{case-1} applied for $(-\varphi$) we have
\begin{equation}
\label{end-stab-eq6}
	(-\hat t_1) - (-\hat t_0) = \hat t_0 - \hat t_1 \leq C \varepsilon^\alpha.
\end{equation}
Since now $\int_{\Omega(t_1)}f \omega^n \geq V_{min}$, it follows that
\[
	Vol_\omega( \Omega(t_1)) \geq \frac{V_{min}^{p^*}}{\|f\|_p^{p^*}}
	\geq \frac{V_{min}^{p^*}}{\max\{\|f\|_p^{p^*}, \|g\|_p^{p^*}\}}:=\delta_2>0.
\]
As in  Proposition~\ref{lap-est} we define
\begin{equation}
\label{end-stab-eq11}
	\delta_{2e} =C c_0 \, \delta_2^\frac{4n-1}{n} \quad \mbox{and} \quad
	t_N = 2^{N-1}(t_1- t_0) + t_0.
\end{equation}
Take $N$ to be the smallest integer which is larger than 
\begin{equation}
\label{end-stab-eq10}
	N_2: = \frac{\int_X e^G \omega^n}{\delta_{2e}} +2.
\end{equation}
If we had $Vol_\omega(X\setminus \Omega(t_N)) \geq \delta_2$,  then Proposition~\ref{lap-est} would lead to a contradiction
\[
	\int_{\{t_0< \varphi \leq t_N\}} \omega_u \wedge \omega^{n-1}
\geq 		(N-1) \delta_{2e} > \int_X e^G \omega^n 
\geq 		\int_X \omega_u \wedge \omega^{n-1}.
\]
Therefore, 
\[
	Vol_\omega(\{\varphi \geq t_N\})
= 	Vol_\omega(X\setminus \Omega(t_N))
\leq 	\delta_2.
\]
By the definition of $\delta_2$ it follows that
\[
	\int_{X\setminus \Omega(t_N)} g \omega^n
\leq 		\|g\|_p \left[ Vol_\omega(X\setminus \Omega(t_N))\right]^\frac{1}{p*} 
\leq 		V_{min}.
\]
Thus, 
\begin{equation}
\label{end-stab-eq15}
	t_N \geq \hat t_1.
\end{equation}
We are ready to finish the proof in the second case. By \eqref{end-stab-eq16} we have
\begin{equation}
\label{end-stab-eq17}
	|\varphi| \leq \hat t_0 - t_0.
\end{equation}
Furthermore, from \eqref{end-stab-eq11} and \eqref{end-stab-eq15} it follows that 
\begin{equation}
\label{end-stab-eq18}
	\hat t_0 - t_0 = (\hat t_0 -  t_N) + (t_N - t_0) 
	 \leq (\hat t_0 - \hat t_1) + 2^{N-1}(t_1 - t_0).
\end{equation}
Combine \eqref{end-stab-eq3}, \eqref{end-stab-eq6} and \eqref{end-stab-eq18} to get
\[
	\hat t_0 - t_0 \leq C \varepsilon^\alpha + 2^{N-1}C\varepsilon^\alpha
	= C(1+ 2^{N-1}) \varepsilon^\alpha.		
\]
It follows from this, \eqref{end-stab-eq10} and \eqref{end-stab-eq17} that
\begin{equation}
\label{stab-case-2}
	\|\varphi\|_\infty \leq C(1+ 2^{N_2}) \varepsilon^\alpha.
\end{equation}
The proof in the second case is completed. Finally, the desired stability estimate follows from 
 \eqref{extra-case-1} and \eqref{stab-case-2}. 
\end{proof}

The above stability result  gives the  uniqueness of continuous solutions.
 
\begin{corollary}
\label{weak-unique}
Suppose that $0< c_0 \leq f \in L^p(\omega^n)$, $p>1$. Then there is a unique $u\in PSH(\omega) \cap C(X)$, $\sup_X u =0$, and unique $c>0$ such that
\[
	\omega_u^n= c f\omega^n.
\]
\end{corollary}

\begin{proof} By Lemma~\ref{const-comp} we have that $c$ is uniquely defined. Hence we may assume $c=1$. Take a smooth sequence $f_j$ such that $f_j \geq \frac{c_0}{2}$ and $f_j$  converges to $f$ in $L^p(\omega^n)$,  as $j\to +\infty$. By the theorem of Tosatti and Weinkove \cite{TW10b} there exists a unique $u_j \in C^\infty(X)$, $ \sup_X u_j =0,$ and a unique constant $c_j>0$ such that
\[
	(\omega + dd^c u_j)^n = c_j f_j \omega^n, \quad \omega + dd^c u_j >0.
\]
We first observe that $c_j \to 1$ as $j \to +\infty$. Indeed, it follows from \cite{KN} that
\[
	\frac{1}{C} \leq c_j \leq C ,
\]
where $C = C(\|f\|_p, X, \omega)>0$. Suppose that there existed a subsequence $c_k \to c \neq 1$ as $k \to + \infty$. Consider
\begin{equation}
\label{cor-u-eq3}
	(\omega + dd^c u_k)^n = c_k f_k \omega^n.
\end{equation}
Since the family $\{u_k \in PSH(\omega): \sup_X u_k =0\}$ is relatively compact in $L^1(\omega^n)$, after passing to a subsequence, still writing $u_k$, we may assume that $\{u_k\}$ is a Cauchy sequence in $L^1(\omega^n)$.  \cite[Corollary 5.10]{KN} implies that $\{u_k\}_k$ is a Cauchy sequence in $C(X)$. Taking the limit on two sides of \eqref{cor-u-eq3} we get, by the Bedford-Taylor convergence theorem,  that
\[
	(\omega+ dd^c w)^n = c f \omega^n = c \,\omega_u^n ,
\]
where $w$ is the limit of $\{u_k\}$. This contradicts the uniqueness of the constant (Lemma~\ref{const-comp}).
Thus we can write 
\[
	(\omega + dd^c u_j)^n = F_j \omega^n, 
\]
where $F_j \geq \frac{c_0}{3}$ and $F_j$ converges to  $f$ in $L^p(\omega^n)$ as $j \to +\infty$. By Theorem~\ref{stability-smooth-lp} we have 
\[
	\|u_j - u\|_\infty \leq C \|F_j - f\|_p^\alpha
\]
for a fixed $0< \alpha< \frac{1}{n+1}$. Since $u_j$ are  unique, we infer that $u$ is also unique.
\end{proof}

\begin{observation}\label{relax-smooth}
Having Corollary~\ref{weak-unique}, the statement of Theorem~\ref{stability-smooth-lp} holds without the smoothness assumption on $f$. 
Thus we obtain Theorem A.
It follows from the  approximation argument as in the proof of Corollary~\ref{weak-unique}.
\end{observation}

 \begin{observation}
The referee provided the following argument which helps to get the stability estimates for continuous solutions Theorem~\ref{stability-smooth-lp}  once they are proven in the smooth category. Let $f, g$ and $u,v$ as in Theorem~A. Let $f_j, g_j$ be smooth densities approximating $f,g$ in $L^p$ with $f_j\geq c_0/2$. Let $u_j,v_j$ be decreasing sequences of smooth $\omega$-psh functions which converge to $u,v$ respectively. Let $\varphi_j, \psi_j$ be smooth $\omega$-psh functions solving 
$$
	(\omega+ dd^c\varphi_j)^n = e^{\varphi_j-u_j}f_j \omega^n, \quad
	(\omega+ dd^c \psi_j)^n = e^{\psi_j-v_j} g_j \omega^n.
$$
By the proof of \cite[Theorem~2.1]{cuong15} we derive that $\varphi_j$ and $\psi_j$ converge uniformly to $u_0, v_0$ respectively, where $u_0, v_0 \in PSH(\omega) \cap C^0(X)$ are unique solutions to 
$$
	\omega_{u_0}^n = e^{u_0} (e^{ - u} f) \omega^n \quad
	\text{and}\quad \omega_{v_0}^n = e^{v_0}(e^{ - v} g) \omega^n,
$$
respectively. By uniqueness of $u_0, v_0$  for the corresponding equations \cite[Lemma~2.3]{cuong15}, we get that $u= u_0$ and $v=v_0$. Therefore,
$$
\|u-v\|_\infty = \lim_{j\to \infty} \|\varphi_j- \psi_j\|_{\infty} \leq \lim_{j\to \infty}C \|F_j - G_j\|_p^\alpha,
$$
where $F_j= e^{\varphi_j-u_j} f_j$ and $G_j = e^{\psi_j -v_j} g_j$ are smooth and positive. The desired stability  easily follows.
\end{observation}


One can prove a variant of the stability theorem, where there is $L^1$ norm on the right hand side,  but then  the exponent is worse by factor $1/n$.

\begin{theorem}
\label{stability-l1}
Let $0\leq f, g\in L^p(\omega^n)$, $p>1$, be such that $\int_X f \omega^n >0$, $\int_X g\omega^n>0$. Consider two continuous $\omega$-psh solutions of the complex Monge-Amp\`ere equation 
\[
	\omega_u^n = f\omega^n, \quad 
	\omega_v^n = g\omega^n
\] 
with 
$
	\sup_X u = \sup_X v =0.
$
Assume that 
\[
	 f \geq c_0>0 .
\] 
Fix $0< \alpha < \frac{1}{2+n(n+1)}$. Then, there exists $C = C(c_0, \alpha, \|f\|_p, \|g\|_p)$ such that
\[
	\|u-v\|_\infty \leq C \|f-g\|_1^\alpha.
\]
\end{theorem}

\begin{observation}\label{bounded-uni}  
 It is, in fact,  enough to assume $v\in PSH(\omega) \cap L^\infty(X)$ as  we are not going to use the minimum principle (Proposition~\ref{min-prin}) in the proof. Therefore, as in Corollary~\ref{weak-unique}, we get the uniqueness of bounded $\omega$-psh solutions for the right hand side in $L^p$.  Strictly speaking we have to use the quasi-continuity of plurisubharmonic functions to get the statement of Lemma~\ref{lap-est-collar} in this case by approximation argument. Again, the following simpler proof is due to the referee. Assume that $u$ is a bounded $\omega$-psh function such that $\omega_u^n =f\omega^n$, where $f\geq 0$ belongs to $L^p(X)$, $p>1$. By \cite[Theorem~2.1]{cuong15} there exists a unique $v\in PSH(\omega)\cap C^0(X)$ such that $$\omega_v^n = e^{v-u} f\omega^n.$$
Since $u$ is another bounded solution to the equation, by uniqueness we get that $u=v$ is also continuous.
\end{observation}

\begin{proof}
The theorem will follow as soon as we prove the following version  of Lemma~\ref{case-1}.
We use again notation:
\[
	\varphi = u-v, \quad \Omega(t) = \{\varphi<t\}, \quad t_0 = \inf_X \varphi.
\]

\begin{lemma}
\label{case1-l1}
Let $V_{min}>0$ be the constant in Proposition~\ref{barrier-funct}. Fix $t_1 > t_0$.
Assume that 
\[
	\|f-g\|_1 \leq \varepsilon
\]
for $0<\varepsilon <<1$. If $\int_{\Omega(t_1)} f \omega^n \leq V_{min}$, then 
\begin{equation}
	t_1 - t_0 \leq C \varepsilon^\alpha 
\end{equation}
where $0< \alpha <\frac{1}{2+ n(n+1)}$ is fixed.
\end{lemma}

\begin{proof}
Define the sets:
\[
\Omega_1 := \{z\in \Omega(t_1): f (z) \leq (1+ \varepsilon^\alpha) g (z)\} \quad
\mbox{and} \quad \Omega_2:= \Omega(t_1)\setminus \Omega_1.
\]
Since $g < \varepsilon^{-\alpha} (f-g)$ on $\Omega_2$, we have
\begin{equation}
\label{c1-l1-eq4}
\begin{aligned}
	\int_{\Omega_2} f \omega^n
&	\leq 		\int_{\Omega_2} |f-g| \omega^n + \int_{\Omega_2} g \omega^n \\
&	\leq		\varepsilon + \varepsilon^{1-\alpha} \leq 2 \varepsilon^{1-\alpha}.
\end{aligned}
\end{equation}
It follows that 
\begin{align*}
	\int_{\Omega(t_1)} f \omega^n
	=	\int_{\Omega_1} f \omega^n + \int_{\Omega_2} f \omega^n
	\leq	\int_{\Omega_1} f \omega^n+ 2 \varepsilon^{1-\alpha}
	\leq V_{min} + 2 \varepsilon^{1-\alpha} .
\end{align*}
We  construct a barrier function which  is a bit  different than the one  in Lemma~\ref{case-1}. Put 
\begin{equation} \label{c1-l1-eq6}
\hat f (z)  =	\begin{cases}
		f (z)\quad &\mbox{for } z\in \Omega(t_1), \\
		\frac{1}{A}f(z)\quad &\mbox{for }  z\in X \setminus \Omega(t_1).
	\end{cases}
\end{equation}
As $\int_{\Omega(t_1)} f\omega^n \leq V_{min}$ we can choose $A>1$ large enough so that 
\[
	\int_X \hat f \omega^n\leq \frac{3}{2} V_{min} .
\]
Notice that $\|\hat f\|_p \leq \|f\|_p$. By \cite[Theorem 0.1]{KN} we find $w \in PSH(\omega) \cap C(X)$ and $\hat c>0$ satisfying
\[
	(\omega + dd^c w)^n = \hat c \hat f \omega^n, \quad \sup_X w =0 .
\]
By  Proposition~\ref{barrier-funct} applied for $A_0= \|f\|_p$  we have 
\begin{equation}\label{c1-l1-eq9}
	2^n \leq \hat c \leq A ,
\end{equation}
where the last inequality follows from \eqref{c1-l1-eq6} and Lemma~\ref{const-comp}.
Hence, 
\begin{equation}\label{c1-l1-eq10}
	\hat c\hat f \geq 2^n f \quad \mbox{on } \Omega(t_1).
\end{equation}
Define for $0< s <1$, 
\[
	\psi_s = (1-s) v + s w.
\]
It follows from the mixed form type inequality (Proposition \ref{mixed-type-ineq}) that 
\begin{equation*}
\begin{aligned}
	(\omega + dd^c \psi_s)^n 
&	\geq 		\left[ (1-s) g^\frac{1}{n} + s (\hat c \hat f/f)^\frac{1}{n}\right]^n \omega^n \\
&	=[(1-s) (g/f)^\frac{1}{n} + s(\hat c \hat f/f)^\frac{1}{n} ]^n f\omega^n \\
&	=: [b(s)]^n f\omega^n. 
\end{aligned}
\end{equation*}
Therefore, on $\Omega_1$, we have
\[
	b(s) \geq \frac{(1-s)}{(1+\varepsilon^\alpha)^\frac{1}{n}} + 2 s 
	\geq \frac{1-s}{1+\varepsilon^\alpha} + 2s.
\]
If $2 \varepsilon^\alpha \leq  s \leq 1$, then 
\begin{equation}\label{c1-l1-eq14}
	b(s) \geq 1 + \varepsilon^\alpha \quad \mbox{on } \Omega_1. 
\end{equation}
Let us use the notation:
\[
	m_s:= \inf_X (u - \psi_s) = \inf_X \{u-v + s(v - w)\} .
\]
We have that 
\[
	m_s \leq t_0 + s \|w\|_\infty .
\]
Set for $0<\tau<1$, 
\[
	m_s(\tau) := \inf_X [u - (1-\tau)\psi_s] .
\]
Then
\[
	m_s(\tau) \leq m_s.
\]
By the above definitions we have
\begin{equation}\label{c1-l1-eq19}
\begin{aligned}
U(\tau, t)
& :=	\{u< (1-\tau)\psi_s + m_s(\tau) + t \}  \\
&	\subset 	\{u < \psi_s + m_s + \tau \|\psi_s\|_\infty + t\} \\
&	\subset	\{u<v + t_0 + s(\|v\|_\infty + \|w\|_\infty) + \tau \|\psi_s\|_\infty+ t \} .
\end{aligned}
\end{equation}
We are going to show that
\begin{equation}\label{c1-l1-eq20}
	t_1 - t_0 \leq 2s(\|v\|_\infty + \|w\|_\infty) + \tau \|\psi_s\|_\infty,
\end{equation}
for $s= 2\varepsilon^\alpha$ and $\tau = \varepsilon^\alpha/2$.  Suppose  it is false. 
By \eqref{c1-l1-eq19} we have
\begin{align*}
	U(\tau, t)  \subset \subset	\{u< v + t_0 + (t_1-t_0)\} = \Omega(t_1), 
\end{align*}
for $0<t < \frac{t_1 - t_0}{2}$. 
To go further we need to estimate the integrals:
\[
	\int_{U(\tau,t)} f\omega^n
\]
for $0< t << s, \tau$.
By the modified comparison principle \cite[Theorem 0.2]{KN} 
\[
	\int_{U(\tau, t)} \omega_ {(1-\tau) \psi_s}^n 
\leq \left(1 + \frac{C t}{\tau^n}\right) \int_{U(\tau, t)} \omega_u^n,
\]
for every $0< t < \min\{\frac{\tau^3}{16B}, \frac{t_1-t_0}{2}\}$. Hence, a simple estimate from below gives 
\[
	(1-\tau)^n \int_{U(\tau,t)} \omega_{\psi_s}^n \leq 
\left(1 + \frac{C t}{\tau^n}\right) \int_{U(\tau, t)} \omega_u^n.
\]
Using \eqref{c1-l1-eq14} for $s=2\varepsilon^\alpha$ we get
\begin{equation}\label{c1-l1-eq25}
	(1-\tau)^n (1+\varepsilon^\alpha)^n \int_{U(\tau, t)\cap \Omega_1} f \omega^n
\leq 	\left(1 + \frac{C t}{\tau^n}\right) \int_{U(\tau, t)} f \omega^n.
\end{equation}
If we write $a(\varepsilon, \tau) = (1-\tau)^n (1+\varepsilon^\alpha)^n$, then
\[
	a(\varepsilon, \tau) = (1+\varepsilon^\alpha/2 - \varepsilon^{2\alpha}/2)^n
	>1 + \varepsilon^\alpha/4
\]
as we have $\tau = \varepsilon^\alpha/2$ and $0<\varepsilon^\alpha<1/4$. Therefore, \eqref{c1-l1-eq25} implies that
\[
	\left[ a (\varepsilon, \tau) - \left(1 + \frac{2^nC t}{\varepsilon^{n\alpha}}\right) 
	\right] \int_{U(\tau,t)\cap \Omega_1} f \omega^n 
\leq  		\left(1 + \frac{2^nC t}{\varepsilon^{n\alpha}}\right)  
		\int_{\Omega_2} f \omega^n .
\]
Thus, for $0< t \leq \varepsilon^{(n+1)\alpha}/2^{n+3}C$,
\[
	\frac{\varepsilon^\alpha}{8} \int_{U(\tau,t)\cap \Omega_1} f \omega^n 
\leq		2 \int_{\Omega_2} f \omega^n 
\leq 		4 \varepsilon^{1-\alpha} ,
\]
where the last inequality used \eqref{c1-l1-eq4}. Hence,
\begin{equation*}
	 \int_{U(\tau,t)\cap \Omega_1} f  \omega^n \leq 32 \, \varepsilon^{1-2\alpha} .
\end{equation*}
Altogether we get that for $0< t \leq \varepsilon^{(n+1)\alpha}/C$, 
\begin{equation}\label{c1-l1-eq30}
\int_{U(\tau, t)} f \omega^n 	
\leq	\int_{U(\tau,t)\cap \Omega_1} f  \omega^n +
	\int_{\Omega_2} f \omega^n
\leq 	C \varepsilon^{1-2\alpha} .
\end{equation}
This is the estimate  we need. Now we are able  make use of the results in \cite{KN}. 
First, it follows from \cite[Theorem 5.3, Remark 5.5]{KN} that
\begin{equation}\label{c1-l1-eq31}
	\hbar (t/2) \leq cap_\omega(U(\tau,t/2)) \leq \frac{2^nC}{t^n} \int_{U(\tau, t)} f\omega^n
\end{equation}
for $0< t \leq \varepsilon^{(n+1)\alpha}/C$, where $\hbar(t)$ is the inverse of $\kappa(t)$ (see Proposition~\ref{barrier-funct}). 
It follows from \eqref{c1-l1-eq30} and \eqref{c1-l1-eq31} that 
\[
	\hbar(t) \leq \frac{C\varepsilon^{1-2\alpha}}{t^n} .
\]
Then, taking $t =\varepsilon^{(n+1)\alpha}/C$ we obtain that
\begin{equation}\label{c1-l1-eq33}
	\hbar (\varepsilon^{(n+1)\alpha}/C) \leq C \varepsilon^{1-2\alpha - n(n+1)\alpha} =: C\varepsilon^\delta ,
\end{equation}
where $\delta= 1 - [n(n+1)+2] \alpha>0$.
However,  we have that
\[
	\hbar(t) \geq \left(\frac{1}{a} \log \frac{C}{t}\right)^{-n} 
\]
for $0<t<<1$ (see \cite[Eq. (3.23)]{cuong15}). As $\delta>0$  is fixed, this leads to a contradiction in the inequality \eqref{c1-l1-eq33}
for $\varepsilon>0$ small enough. Thus we have proved that 
\[
	t_1 - t_0 \leq 4 \varepsilon^{\alpha} (\|v\|_\infty + \|w\|_\infty + \|\psi_s\|_\infty),
\]
for a fixed $0< \alpha < \frac{1}{2 + n(n+1)}$. The  norms on the right hand side are  controlled by $\|f\|_p, \|g\|_p, A, V_{min}$. So the  lemma follows.
\end{proof}

The rest of  the proof of Theorem~\ref{stability-l1} is analogous to that  of Theorem~\ref{stability-smooth-lp}.
Notice that the proof of Lemma~\ref{led-lower-bound}, without the smoothness  assumptions on $f,u$, follows by the approximation argument (as in Corollary~\ref{weak-unique}) because  the uniqueness of continuous solutions has been proven for $u$. Again, by the approximation argument, if $v\in PSH(\omega) \cap L^\infty(X)$ then it  is enough to get the inequality \eqref{lec-eq31} (see \cite[Theorem 3.2]{BT87}).
\end{proof}




\section{H\"older continuity}
\label{S3}

In this section we prove the H\"older continuity of solutions of   the complex Monge-Amp\`ere equation
 $$\om _u^n = f\om ^n ,$$
when $f\in L^p (X, \om ), \  p>1$ and $f>0.$ The H\"older exponent is explicitly given in terms of $p$ and the dimension $n$.
In the K\"ahler case, with $f$ only nonnegative, the result is due to the first author \cite{kolodziej08} (with exponent depending on the manifold),
 and   to  Demailly, Dinew, Guedj, Hiep, Ko\l odziej, Zeriahi \cite{demailly-et-al} in full generality. 
The new ingredient in the latter proof was the application of Demailly's regularization method for $\om$-psh functions,
which uses the Riemann exponential map (\cite{De82}) or the holomorphic part of this map (\cite{De94}) - suitable for non-K\"ahler manifolds.
We shall follow this scheme with necessary modifications. Namely, to apply Theorem~A for a small perturbation of $f$ we need one more lemma (Lemma \ref{stab}), since the perturbation defined as on K\"ahler manifold is, generically, not MA-admissible. This makes the H\"older exponent worse, by factor $1/n.$
Furthermore, the comparison principle used in the K\"ahler case is not available here, so we change the proof  to apply just the minimum  principle.

Following \cite{De94} consider $\rho_{\delta }u$- the regularization of the $\omega$-psh function $u$ defined  by

\begin{equation}\label{phie}
\rho_\delta u(z)=\frac{1}{\delta ^{2n}}\int_{\zeta\in T_{z}X}
u({\exp} h_z(\zeta))\rho\Big(\frac{|\zeta|^2_{\omega }}{\delta ^2}\Big)\,dV_{\omega}(\zeta),\ \delta>0;
\end{equation}
where $\zeta \to {\exp}h_z(\zeta)$ is the (formal) holomorphic part of the Taylor expansion of the exponential map of  the Chern connection on the tangent bundle of $X$   associated to $\omega $, and the smoothing kernel 
 $\rho: \mathbb R_{+}\rightarrow\mathbb R_{+}$  is given by
$$\rho(t)=\begin{cases}\frac {\eta}{(1-t)^2}\exp(\frac 1{t-1})&\ {\rm if}\ 0\leq t\leq 1,\\0&\
{\rm if}\ t>1\end{cases}$$
 with a suitable constant $\eta$, such that
\begin{equation}\label{total integral}
\int_{\mathbb C^n}\rho(\Vert z\Vert^2)\,dV(z)=1
\end{equation}
($dV$ being the Lebesgue measure in $\mathbb C^n$).

Of crucial importance is  the following lemma  from \cite[ Proposition 3.8]{De94}, and  \cite[Lemma 1.12]{BD12}. For the sake of completeness we include its proof.

\begin{lemma}\label{kis}
 Fix any bounded $\omega$-psh function $u$ on a compact Hermitian manifold $(X, \omega )$.  Define the Kiselman-Legendre transform with level $c>0$ by
 \begin{equation}\label{kisleg}
 U_{\delta , c  }= \inf _{ t\in [0,\delta ]}(\rho_{t }u + K(t^2-\delta^2) + K(t- \delta) -c \log\frac{t}{\delta }),
 \end{equation}
Then for  some positive constant $K$  depending on the curvature, the function $\rho_{t }u+Kt^2$ is increasing in $t$ and
 the following estimate holds:
\begin{equation}\label{hessest}
\omega+dd^c  U_{\delta,c}\geq -(A c+2K\delta)\,\omega,
\end{equation}
where $A$ is a lower bound of the negative part of the Chern curvature of $\omega$.
\end{lemma}

\begin{proof}
Note that $ U_{\delta , c  }$ defined here differs from that in \cite[Lemma 1.12]{BD12} by the term $K(t - \delta)$ (as in \cite[ Remark 4.7]{De94}).
The upshot of adding it is that the Laplacian of $|w|$  is bigger than $\frac{1}{4  \delta }$ for $|w|\leq \delta $
and  one can use this  in the Cauchy-Schwarz inequality to estimate one of the terms in the inequality \cite[Eq. (1.8)]{BD12}. Namely, 
\begin{equation}\label{eq:cs-extra-term}
|dz||dw|\leq \frac{1}{4\delta}|dw|^2 +  \delta |dz|^2.
\end{equation}
Since $u$ is bounded, its Lelong number at every $z\in X$ is zero, then by \cite[Eq. (1.9)]{BD12}
\[ \lambda(z,t) = \frac{\partial (\rho_t u(z) + Kt^2)}{ \partial \log t} \to 0 \mbox{ as } t \to 0^+.
\]
Moreover, the upper-level set of the Lelong numbers $\{\nu (u, z)\geq c \}$ is empty for any fixed $c>0$. Therefore, for $z\in X$ the infimum  in \eqref{kisleg} is  attained for $t= t_0(z)>0$. More precisely, $t_0(z) = \delta$ if 
\[
	\lambda(z, \delta) + K \delta \leq c,
\]
and otherwise $0< t_0(z) < \delta$ satisfying (zero of the $\partial/\partial \log t$ derivative):
\[
	\lambda(z, t_0) + K t_0  - c = 0.
\]
The implicit function theorem shows that $t_0(z)$ depends smoothly on $z$. Hence,
$U_{\delta,c} (z)$ is smooth on $X$. 

Now fix a point $x\in X$ and $t_1 > t_0(x)$. For all $z$ in a neighbourhood $V$ of $x$ we still have $0< t_0(z) < t_1$, thus
\[
U_{\delta,c} (z) = \inf_{0<|w|< t_1} \left(U(z,w) + K(|w|^2 - \delta ^2 ) +  K(|w| - \delta) -c \log\frac{|w|}{\delta } \right) \mbox{ on }  V,
\]
where $U(z,w) = \rho_{t}u(z)$ for $t=|w|$. We  are going to get the lower bound on the set $V\times \{ w : 0<|w| < t_1\}$ of the complex Hessian in $(z,w)$  of the function on the right hand side of the last formula.

By \cite[Eq. (1.8)]{BD12} and \eqref{eq:cs-extra-term} we have
\begin{equation}
\label{eq:demailly-estimate}
\begin{aligned}
&	\omega(z) + dd^c U(z,w) \\
&\geq - A \lambda(z, |w|) |dz|^2 - K\left(|w|^2 |dz|^2 +|dz||dw| + |dw|^2\right) \\
&\geq - A \lambda(z,|w|) |dz|^2 - K (|w|^2 + \delta) |dz|^2 - K(1+ \frac{1}{4\delta}) |dw|^2.
\end{aligned}\end{equation}
An easy computation gives  for $|w| \leq \delta$,
\begin{equation}\label{eq:add-1} dd^c K(|w|^2 + |w|) = K(1 +\frac{1}{4|w|}) i dw\wedge d\bar w \geq K(1 + \frac{1}{4\delta}) |dw|^2.
\end{equation}
Since $\lambda(z,t_0) < c$, and is increasing in $t$, we have that 
\begin{equation}\label{eq:add-2}
\lambda(z, |w|) \leq c \quad \mbox{on }V\times\{w: 0< |w| <t_1\}
\end{equation}
 (decreasing $t_1$ and shrinking $V$ if necessary). Lastly, $-c \log|w|$ is pluriharmonic for $|w|>0$. Combining \eqref{eq:demailly-estimate}, \eqref{eq:add-1} and \eqref{eq:add-2}  we obtain that on $V \times \{w: 0< |w| <t_1\}$ 
\[
\begin{aligned}
\omega + dd^c \left(U(z,w) + K(|w|^2 - \delta ^2 ) +  K(|w|- \delta) -c \log\frac{|w|}{\delta} \right) 
\geq & -(Ac +2K\delta )\omega.
\end{aligned}
\]
Now the desired bound \eqref{hessest} follows from Kiselman's minimum principle \cite{kis}. 
\end{proof}

We also need a lemma from \cite{demailly-et-al} (which is actually stated for K\"ahler manifolds, but the same proof works for Hermitian ones
after replacing the Riemann curvature tensor by the Chern curvature tensor used in  \cite{De94})
saying  that for some constant $C$ depending on $\om $ and $\|u\|_\infty$
\begin{equation}\label{jen}
\int _X \frac{|\rho_{t }u -u|}{t^2} \om ^n <C,
\end{equation}
for any $\om$-psh function $u$ and $t$ small enough.

\begin{theorem}\label{holder}
 Consider the  solution $u$ of the complex Monge-Amp\`ere equation
$$\om _u^n = f\om ^n , $$
on  $(X, \om )$  a compact  Hermitian manifold with $ f >c_0 > 0, \|f\|_p <\infty.$
Then for any    $\alpha < \frac{2}{p^* n (n+1)+1}$    the function $u$
is H\"older continuous, with H\"older exponent $\alpha .$
\end{theorem}
\begin{proof}  
As explained in  \cite{kolodziej08} and \cite{demailly-et-al} the result follows as soon as we show
that
$$
\rho _t u - u \leq ct^{\al }
$$
for $t$ small enough.

\begin{lemma}\label{stab}
For $p>q>1$ set $r=\frac{npq}{nq+p-q}.$
Suppose $f\in L^p (X)$  is MA-admissible and for a Borel set $E$ and sufficiently small positive $\de $ one has
 $\int _E f^q\om ^n <\de ^q$ and $ \int _E f^r\om ^n <\de ^{\frac{q(p-r)}{p-q} }    .$ Consider a MA-admissible perturbation of $f$
\begin{equation}\label{g}
g (z) = \begin{cases}
  0 \ \ \  \ \ \ \ \  \ \ \  \ \ \ \ \  \text{for}  \ z\in E \\
  (1+s ) f(z)  \ \ \ \  \text{for}    \ z\in X\setminus E.
 \end{cases}
\end{equation}
Then  $\|f-g\|_{r } <C \de ^{1/n}$  for some $C>0$ independent of $\de .$
\end{lemma}
\begin{proof} Notice that $\int_{X\setminus E} f \omega^n>0$ for small $\delta>0$. Therefore, by \cite[Theorem~0.1]{KN} there exist $b:= 1+s>0$ and $v\in PSH(\omega)\cap C^0(X)$ solving
\[\label{eq:ma-admissible-g}
	\omega_v^n = b {\bf 1}_{X\setminus E} f\omega^n.
\]
By Lemma~\ref{const-comp} it is clear that $s \geq 0$. In other words, $g$ is MA-admissible. 

First, we are going to show that 
\[	s \leq C \delta^\frac{1}{n} 
\]
with $C>0$ independent of $\delta$. Indeed, we define for $N>1$:
\begin{equation}\label{h}
h (z) = \begin{cases}
  \de^{-1}V_{min} \, f(z)   \  \ \ \   &\text{for}  \ z\in E \\
  \frac{ 1}{N} \,f(z)  \ \ \ &\text{for}    \ z\in X\setminus E,
 \end{cases}
\end{equation}
where $V_{min}$  (defined in  Proposition \ref{barrier-funct})  corresponds to the norm  of $f$  in $L^q(\omega^n)$. Using our assumptions and $Vol_\omega(X)=1$ we get that
\[
	\int_X h \omega^n \leq V_{min} + \frac{1}{N} \|f\|_1, \quad
	\int_X h^q \omega^n \leq V_{min}^q + \frac{1}{N^q} \|f\|_q^q.
\]
Let $w \in PSH(\omega) \cap C(X)$, $c>0$ solve
\[
	\omega_w^n = c \, h \omega^n.
\]
By Proposition~\ref{barrier-funct} we have $c \geq 2^n$ for $N>1$ large enough. 
Set for $0< t <1$,
$$\psi = (1- t)v+ tw.$$
By the mixed form type inequality 
$$
(\om ^n _{\psi }/\omega^n )^{1/n} \geq (1-t)g^{1/n} + t c^{1/n}h^{1/n}.
$$
The right hand side exceeds  $$ 2 t V_{min}^{1/n}\de ^{-1/n}f^{1/n}$$ on $E$ and 
$$(1-t)(1+s)^{1/n}f^{1/n}$$ on $X\setminus E .$ Thus, $(\om ^n _{\psi }/\omega^n )^{1/n}$ is strictly bigger than $f^{1/n}$ on $X$ for 
\begin{equation}\label{g-eq1}
	t \geq \delta^{1/n} V_{min}^{-1/n}  \quad \mbox{and} \quad
	(1-t)(1+s)^{1/n} >1.
\end{equation}
The latter inequality is equivalent to
\[
t < \frac{(1+s)^{1/n}-1}{(1+s)^{1/n}} = \frac{s}{\sum_{k=1}^n (1+s)^{k/n}}.
\]
Note that $0\leq s \leq C(\|f\|_p, X, \omega)$. Then, $t = s/C$ satisfies the second inequality in \eqref{g-eq1} for $C>0$ large enough (independent of $\delta$). Therefore, Lemma \ref{const-comp} implies that the first inequality in \eqref{g-eq1} cannot hold for $C>0$ large, as otherwise we would have $$\omega_\psi^n > c_0 \omega_u^n$$
for some $c_0>1$. Thus, we get that
\begin{equation}\label{g-eq2}
	s < C \de^{1/n} V_{min}^{-1/n}.
\end{equation}
Hence \eqref{g-eq2} implies that
$$
\|f-g\|_r^r \leq \left( \int _E f^r\om ^n +  s^r \|f\|_r^r \right) \leq C\max (\de ^{\frac{q(p-r)}{p-q} }, \de ^{r/n})
$$
and by our choice of $r$ the exponents on the right hand side are equal. Therefore $\|f-g\|_r \leq C\de ^{1/n}$,
which gives the statement.
\end{proof}


 Suppose that  $A-1 > 0$ is a bound for the curvature of $(X, \omega)$ from the statement of Lemma~\ref{kis}.
 By \cite{KN} $u$ is continuous, so assume that $\min _X  u =1$ and put $b := 2\max_X u$ and $\theta := e^{-5Ab}$.

Fix $\alpha < \frac{2}{p^* n(n+1)+1}$  and  note that it is equivalent to 
\begin{equation}
\alpha < \frac{2-\al }{p^*  n (n +1)}.
\end{equation}
Therefore one can choose $q>1$ so close to $1$ that  the following inequality holds
\begin{equation}\label{2}
\alpha < \frac{(2-\al )(p-q)}{pq n (n +1)}.
\end{equation}
 Let  us set for $\delta > 0$:
 \begin{equation}\label{2new}
 E(\delta ) := \{(\rho_{\delta }u-u)(z) \geq  Ab \delta ^{\alpha}\} .
 \end{equation}
By the definition of the Kiselman-Legendre transform at level $\delta^{\alpha}$ (see  \eqref{kisleg})
 $$
 U_{\delta }= \inf _{ t\in [0,\delta ]} \left(\rho_{t }u + K(t^2 -\delta^2) + K(t-\delta) -\delta ^{\alpha }\log\frac{t}{\delta } \right),
 $$
 where $K$ is chosen as in the formula \eqref{kisleg}.
Recall from Lemma \ref{kis} 
that the functions $\rho_{\delta }u+K\delta^2$ are increasing in $\delta$.
By the same lemma 
 $$
 \omega +dd^c U_{\delta } \geq -[(A-1)\delta ^{\alpha }+ 2K\delta ]\omega ,
 $$
 for $0 < \delta  < \delta_0$, where $\delta_0 > 0$ is small enough. We can also assume that 
 \begin{equation}
\label{eq:small-alpha0}
 A\delta_0 ^{\alpha}<1  \quad \mbox{and}\quad  2K\delta _0  ^{1-\al}<Ab,
  \end{equation}
which will be used later.
Therefore
 $$u_{\delta }:=\frac{1}{1+A\delta ^{\alpha }  } U_{\delta }$$
 is  $\omega$-psh on $X$ and satisfies
 $$
 \omega + dd^c u_{\delta} \geq \frac{1}{2}\delta^{\alpha } \omega.
 $$
Note that the definitions of $ U_{\delta }$  and $\theta$ lead to
$$
\inf _{ t\in [0, \theta\delta ]} \left(\rho_{t }u + K(t^2 - \delta^2) + K(t-\delta)-\delta ^{\alpha }\log\frac{t}{\delta }\right)\geq u-K(\de +\de^2) +5Ab\de ^{\al },
$$
and therefore $U_{\delta}$ is larger than
\[
\min \left\{\inf _{ t\in [\theta\delta,\delta ]}\left(\rho_{t }u + K(t^2 -\delta^2) +K(t - \delta)-\delta ^{\alpha }\log\frac{t}{\delta }\right), 
u-K(\de+\de ^2) +5Ab\de ^{\al }\right\}.
 \]
By monotonicity of $\rho_{t}u+Kt^2$ one infers
$$
 U_{\delta } \geq \min \left\{ \rho _{\theta\delta } u -K(\de+ \de ^2), u-K(\de+\de ^2) +5Ab\de ^{\al }\right\}.
$$
Thus, combining this with \eqref{eq:small-alpha0} we have on the set
$$
F(\de ) = \{  \rho_{\theta\delta } u-u \geq 5Ab\de ^{\al } \}$$
that
 \begin{equation}\label{3}
U_{\de } -u \geq 4Ab\de ^{\al }.
\end{equation}
To prove H\"older continuity of $u$ with the exponent $\al $ it is enough to show that $F(\de )$ is empty for $\de$ small enough.
Reasoning by contradiction, we assume that  $F(\de )\neq\emptyset  $.
From \eqref{jen}  we have
 \begin{equation}\label{used}
 \int _X \vert \rho_{\delta }u -u\vert \omega ^n \leq c_1 \delta ^2 , \ \
 \end{equation}
 for $0 < \delta < \delta_0$ (decreasing $\de _0$ if needed).
 Therefore   
 $$
 \int _{E (\de )} \omega ^n \leq \frac{c_1}{Ab} \delta ^{2-\alpha }.
 $$
 Hence, by H\"older's inequality,
 $$
 \int _{E(\de )} f^q \omega ^n < c_2 \delta ^{(2-\alpha ) (p-q)/p }
 $$ 
and
 $$
 \int _{E(\de )} f^r \omega ^n < c_2 \delta ^{(2-\alpha  )(p-r)/p },
 $$
for $r<p.$
Let us define $g$ as in Lemma \ref{stab} with an open set $E$ containing $E(\de )$ and such that 
 $$
 \int _{E} f\omega ^n < c_2 \delta ^{(2-\alpha  )(p-q)/p} 
 $$ 
and
$$
 \int _{E} f^r \omega ^n < c_2 \delta ^{(2-\alpha  )(p-r)/p } .
 $$
Then Lemma \ref{stab} implies  $\|f-g\|_r  \leq c_3    \delta ^{(2-\alpha  )(p-q)/npq }.$
Our stability theorem (and (\ref{2})) give, for $v$ defined by the equations
 $$
 (\omega +dd^c v)^n = g\omega ^n ,\ \ \ \max (u-v)=\max (v-u),
 $$
that
$$
 \| u-v\|_{\infty}  \leq c_4 \delta ^{\alpha ' },
$$
for some $\alpha ' >\alpha.$
The constants $c_j$ are independent of $\de .$ Decreasing $\de _0$ , we finally obtain
 \begin{equation}\label{4}
 \| u-v\|_{\infty}  \leq \frac{Ab}{2} \delta ^{\alpha }, \ \ \ \de <\de _0.
 \end{equation}
 Observe that the choice of $b$ gives
 $$
 U_{\delta }- u_{\delta }\leq \frac{Ab}{2} \delta ^{\alpha }.
 $$
This inequality combined with (\ref{3}) and (\ref{4}), for $z\in F(\de ),$  leads to
$$
(u_{\de } -v)(z ) = [(u_{\de } -U_{\de } ) +(U_{\de } -u)+(u-v)](z)\geq 3Ab\de ^{\al }.
$$
On the other hand, if $g(z)>0$, then $z\notin E$ and therefore $(U_{\de }-u)(z) <  Ab\de ^{\al }$.
Again, applying  (\ref{4}), we obtain
$$
(u_{\de } -v)(z )\leq 2Ab\de ^{\al }.$$
The last two estimates prove that $\max (u_{\de } -v)$ is attained within the  open set $E$. 
However on this set $\om _v ^n =0 < \om _{ u_{\de } } ^n ,$  which contradicts the minimum principle (Proposition \ref{min-prin}).
The theorem thus follows.
\end{proof}

\section{ Sz\'ekelyhidi-Tosatti theorem  on  Hermitian manifolds}

 Sz\'ekelyhidi and Tosatti \cite{SzTo11}  considered  a weak solution to the equation
\[
	(\omega + dd^c u)^n = e^{-F(u,z)}\omega^n
\]
on a compact $n$-dimensional K\"ahler manifold $(X, \omega)$. They proved that if $F(x, z)\in C^\infty(\mathbb R \times X)$ and $u \in PSH(\omega)\cap L^\infty(X)$, then $u$ is smooth. An interesting corollary to this result says that
if $X$ is a Fano manifold and  $\omega$ represents the first Chern class, then setting $F(u,z)=  u-h$ with $h$ satisfying $dd^c h= Ric (\omega ) -\omega$, one  concludes  that a  K\"ahler-Einstein current with bounded potentials is  smooth. Nie [Nie13] recently generalised the result of \cite{SzTo11} proving the same
for  $(X,\omega)$  a compact Hermitian manifold  satisfying 
\begin{equation}
\label{vol-pre-con}
	\int_X \omega^n = \int_X (\omega + dd^cu)^n 
	\quad \forall u \in PSH(\omega)\cap C^\infty(X) .
\end{equation}
This is a restrictive assumption, but
as remarked by Nie \cite[Remark 4.1]{nie13},  the only missing ingredient to remove \eqref{vol-pre-con}
 is the stability theorem for the Monge-Amp\`ere equation on compact Hermitian manifolds. Thanks to higher order estimates  in \cite{nie13} and our results in Section~\ref{S2}  we obtain the  proof of  the Sz\'ekelyhidi-Tosatti theorem  on any  compact Hermitian manifold.
 
\begin{theorem}
\label{regularity-thm}
Let $(X,\omega)$ be a compact $n$-dimensional Hermitian manifold. 
Suppose that $u\in PSH(\omega)\cap L^\infty(X)$ is a
solution of the equation
\[ 
	(\omega + dd^c u)^n = e^{-F(u,z)}\omega^n
\]
in the weak sense of currents, where $F:\mathbb{R}\times X\to \mathbb{R}$ is smooth. 
Then $u$ is smooth.
\end{theorem}

\begin{proof}
Without loss of generality we may assume that $\sup_X u =0$. As $u$ is bounded the right hand side is bounded and strictly positive. Hence, by Remark~\ref{bounded-uni}, results in Sections~\ref{S2} and  ~\ref{S3} we know that $u$ is   H\"older continuous  because it coincides with the unique continuous $\omega$-psh solution to the equation 
\[
	(\omega + dd^c u)^n = e^{-F(u,z)}\omega^n, \quad
	\sup_X u =0.
\]

As stated in  \cite[Remark 4.1]{nie13} the argument in this paper  gives the theorem  as soon as the following stability estimate is proven.

\begin{lemma}
\label{cj-1}
Let $u_k$ be smooth functions such that 
\[
	\lim_{k\to +\infty} \|u -u_k\|_\infty =0 .
\]  
By  \cite{TW10b} there exists a unique $\psi_k \in PSH(\omega) \cap C^\infty(X)$ and a unique $c_k >0$ solving
\begin{align*}
	(\omega + dd^c \psi_k)^n = c_k e^{-F(u_k,z)}\omega^n, \quad 
	\sup_X \psi_k =0. 
\end{align*}
Then, 
\[	
	\lim_{k \to +\infty} c_k =1 \quad \mbox{and} \quad
	\lim_{k \to +\infty} \|\psi_k - u\|_\infty =0.
\]
\end{lemma}

\begin{proof}
Observe that $e^{-F(u_k,z)}$ converges uniformly to $e^{-F(u,z)}$. Therefore, the first assertion follows from the proof of Corollary~\ref{weak-unique}. Thus,  we have $\lim_{k \to +\infty} c_k =1$.  It again implies  that $c_k e^{-F(u_k,z)} \in C^\infty(X)$ converges uniformly to $e^{-F(u,z)}$. Therefore the second assertion follows from Theorem~\ref{stability-smooth-lp}.
\end{proof}

Having this lemma the rest of the proof follows the lines of the proof  in \cite[Sec. 4]{nie13}.
\end{proof}

\bigskip

{\noindent Faculty of Mathematics and Computer Science,
Jagiellonian University 30-348 Krak\'ow, \L ojasiewicza 6,
Poland;\\ e-mail: {\tt Slawomir.Kolodziej@im.uj.edu.pl}}\\ \\

{\noindent Faculty of Mathematics and Computer Science,
Jagiellonian University 30-348 Krak\'ow, \L ojasiewicza 6,
Poland;\\ e-mail: {\tt Nguyen.Ngoc.Cuong@im.uj.edu.pl}

\end{document}